\documentclass[11pt]{amsart}
\usepackage{amsmath,amsthm,amssymb,mathtools}
\usepackage{comment}

\newtheorem{theorem}{Theorem}[section]
\newtheorem{lemma}[theorem]{Lemma}
\newtheorem{proposition}[theorem]{Proposition}
\newtheorem{cor}[theorem]{Corollary}

\theoremstyle{definition}

\theoremstyle{remark}
\newtheorem{remark}[theorem]{Remark}

\numberwithin{equation}{section}

\newcommand{\R}{\mathbb{R}}
\newcommand{\C}{\mathbb{C}}
\newcommand{\g}{\mathfrak{g}}
\newcommand{\kk}{\mathfrak{k}}
\newcommand{\p}{\mathfrak{p}}
\newcommand{\m}{\mathfrak{m}}
\newcommand{\ttt}{\mathfrak{t}}
\newcommand{\aaa}{\mathfrak{a}}
\newcommand{\I}{\sqrt{-1}}
\newcommand{\E}{\tilde{E}}

\allowdisplaybreaks[3]

\begin{document}
\title[Matrix-valued commuting differential operators]{Matrix valued commuting differential operators with $A_2$ symmetry}
\author{Nobukazu Shimeno}
\address{School of Science \& Technology  
Kwansei Gakuin University, 
2-1 Gakuen, Sanda 669-1337, Japan 
}
\email{shimeno@kwansei.ac.jp}
\maketitle

\begin{abstract}
We study the algebra of invariant differential operators on a certain homogeneous 
vector bundle over a Riemannian symmetric space of type $A_2$. 
We computed radial parts of its generators explicitly to obtain matrix-valued 
commuting differential operators with $A_2$ symmetry. 
Moreover, we generalize the commuting differential operators with respect to a 
parameter and the  potential function. 
\end{abstract}

\section*{Introduction}
In this paper, we give matrix-valued commuting differential operators 
with $A_2$ symmetry that come from radial parts with respect to 
a certain $K$-type on a real semisimple Lie group of  rank 2. 
Moreover, we give some generalization of commuting differential operators.

First we give a brief outline of the scalar case, which motivates our study. 
Let $G$ be a connected noncompact real semisimple Lie group with finite center, 
$K$ be a maximal compact subgroup of $G$, and $G=KAN$ be 
an Iwasawa decomposition.  Harmonic analysis on 
the Riemannian symmetric space $G/K$ of the noncompact type has 
been extensively studied (\cite{Hel}). In particular, the algebra of the 
left $G$-invariant differential operators on $G/K$ is commutative and 
$K$-invariant joint eigenfunctions, which are called zonal spherical functions 
are important. Heckman and Opdam constructed commuting 
differential operators by allowing the root multiplicities in the radial part of the 
Laplace-Beltrami 
operator on $G/K$ to be continuous parameters. 
Moreover, they constructed a real analytic joint eigenfunction, which is a generalization of   
the radial part of the zonal spherical function (\cite[Part I]{HS}, \cite{Opdam1}). 
By a gauge transformation, the second order operator that is the radial part 
of the Laplace-Beltrami operator in group case becomes a Schr\"odinger operator 
without first order terms and the commuting differential operators containing it 
give a completely integrable system with Weyl group symmetry, 
which is called the Calogero-Sutherland model. The potential function is given by 
 trigonometric function $1/\sinh^2$. 
It is known that there are quantum integrable models with elliptic potential 
functions (\cite{Oint, R}). 

Vector-valued functions naturally arise in group case, if we consider $K$-types of higher 
dimensions. Indeed, $K$-finite matrix coefficients of representations of $G$ 
such as principal series representations 
satisfy differential equations coming from the universal enveloping algebra 
$U(\g_\C)$. 
Let $\tau$ be an irreducible representation 
of $K$ and $E_\tau$ be the associated homogeneous vector bundle over $G/K$, 
and $\mathbb{D}(E_\tau)$ be the algebra of the left $G$-invariant 
differential operators on $E_\tau$. Deitmar \cite{D} proved that $\mathbb{D}(E_\tau)$ 
is commutative if and only if $\tau|_M$ is multiplicity-free, where 
$M$ is the centralizer of $A$ in $K$. 
Even when $\mathbb{D}(E_\tau)$ is commutative, it seems to be hard to 
understand its structure and representations except the 
case of one-dimensional $K$-types. 

In this paper, we consider 
the case of $G=SL(3,\mathbb{K})$ $(\mathbb{K}=\R,\,\C,\,\mathbb{H})$ 
and $\tau$ is the standard representation 
of $K$. Then $\mathbb{D}(E_\tau)$ is commutative and 
the Weyl group $S_3$ acts transitively on constituents of $\tau|_M$. 
In this case, $\mathbb{D}(E_\tau)$ is easy to understand and 
we give radial parts of its generators explicitly. Moreover, 
we generalize these matrix-valued commuting differential operators 
by allowing the root multiplicity to be continuous parameter and 
also to the case of elliptic potential function. 

This paper is organized as follows. In Section 1, we review on some 
known facts on the algebra of invariant differential operators on 
a homogeneous vector bundle over a symmetric space. 
In Section 2, we study invariant differential operators on vector bundles over symmetric 
spaces of type $A_2$. In Section 3, we give 
generalizations of matrix-valued commuting differential operators. 

\section{Invariant differential operators on a homogeneous vector bundle over a 
Riemannian symmetric space}

In this section, we review on the algebra of invariant differential operators on a homogeneous vector bundle over a 
Riemannian symmetric space after \cite{D, Dix, M}. 
\subsection{Notation}

Let $G$ be a connected noncompact real semisimple Lie group 
and $K$ be a maximal compact subgroup of $G$. Let $\g$ and $\kk$ denote the 
Lie algebras of $G$ and $K$ respectively. For a Cartan involution $\theta$ of $\g$ such that $\g^\theta=\kk$, 
let $\mathfrak{g}=\mathfrak{k}+\mathfrak{p}$ be the corresponding Cartan decomposition. 
Let $G=KAN$ be an Iwasawa decomposition and $\g=\kk+\aaa+\mathfrak{n}$ be the corresponding decomposition. 
Let $\Sigma=\Sigma(\g,\aaa)$ denote the restricted root system  and $\Sigma^+$ denote the positive system corresponding to $\mathfrak{n}$. 
For $\alpha\in \Sigma$ let $\g^\alpha$ denote the root space for $\alpha$ and $m_\alpha=\dim\g^\alpha$. 
Put $\rho=\frac12\sum_{\alpha\in\Sigma^+}m_\alpha\alpha$. 
Let $M'$ and $M$ be the normalizer and centralizer of $\aaa$ in $K$ respectively. Then the Weyl group $W$ of $\Sigma$ is isomorphic to $M'/M$. 

For a real vector space $\mathfrak{u}$ let $\mathfrak{u}_\mathbb{C}$ denote its complexification. 
Let $U(\mathfrak{g}_\C)$ denote the universal enveloping algebra of $\g_\C$ and $Z(\g_\C)$ its center. 
Let $U(\g_\C)^K$ denote the set of the $K$-invariants in $U(\g_\C)$.

\subsection{The algebra of invariant differential operators}

Let $(\tau,V_\tau)$ be an irreducible representation of $K$ and $E_\tau$ be the homogeneous vector bundle over $G/K$ 
associated with $\tau$. The space of $C^\infty$-sections of $E_\tau$ is identified with a subspace of the 
$V_\tau$-valued $C^\infty$-functions on $G$:
\[
C^\infty(E_\tau)\simeq\{f\in C^\infty (G,V_\tau)\,:\,f(gk)=\tau(k^{-1})f(g)\,\,(g\in G,\,k\in K)\}.
\]
The action of $G$ on $C^\infty(E_\tau)$ is defined by 
\[
l(g)f(x)=f(g^{-1}x)\quad (f\in C^\infty(E_\tau),\,\,g,\,x\in G).
\]
Let $\mathbb{D}(E_\tau)$ denote the algebra of differential operators $D\,:\,C^\infty(E_\tau)\rightarrow C^\infty(E_\tau)$ 
that satisfy $D\circ l(g)=l(g)\circ D$ for all $g\in G$. We call an element $D\in\mathbb{D}(E_\tau)$ an invariant 
differential operator on $E_\tau$. 

Let $\top$ denote the canonical anti-automorphism of $U(\g_\C)$ defined by $1^\top=1,\,X^\top=-X,\,(XY)^\top=Y^\top X^\top \,\,(X\in\g_\C)$. 
Let $I_\tau$ denote the kernel of $\tau$ in $U(\kk_\C)$. Then we have
\[
\mathbb{D}(E_\tau)\simeq U(\g_\C)^K/(U(\g_\C)^K\cap U(\g_\C)I_\tau^\top).
\]

By the Iwasawa decomposition  and the Poincar\'e-Birkoff-Witt theorem, we have
\[
U(\g_\C)=U(\aaa_\C)U(\kk_\C)\oplus{\mathfrak{n}}_\C U(\g_\C).
\]
Let $p$ denote the projection from $U(\g_\C)$ to $U(\aaa_\C)U(\kk_\C)$. 
Then $p|_{U(\g_\C)^K}$ maps $U(\g_\C)^K$ to $U(\aaa_\C)\otimes U(\kk_\C)^M$. 
Let $\eta$ denote the automorphism of $U(\aaa_\C)$ defined by $\eta(H)=H+\rho(H)\,\,(H\in\aaa_\C)$. 
Then the homomorphism $\gamma_\tau=(\eta\otimes (\tau\circ \top))\circ p$ from 
$U(\g_\C)^K$ to $ (U(\aaa_\C)\otimes \text{End}_M(V_\tau))^{M'}$ induces 
the following injective algebra homomorphism, which we denote by the 
same notation:
\begin{equation}\label{eqn:ghch}
{\gamma}_\tau\,:\,\mathbb{D}(E_\tau)\rightarrow (U(\mathfrak{a}_\C)\otimes
 \text{End}_M(V_\tau))^{M'}.
\end{equation}
The above homomorphism is not necessarily surjective. 

The algebra $\mathbb{D}(E_\tau)$ is not necessarily commutative. If $\tau|_M$ decomposes into multiplicity free 
sum of irreducible representations of $M$, then $U(\mathfrak{a}_\C)\otimes
 \text{End}_M(V_\tau)$ is commutative by Schur's lemma, hence $\mathbb{D}(E_\tau)$ is commutative. 
Deitmar \cite{D} proved that  $\mathbb{D}(E_\tau)$ is commutative if and only if $\tau|_M$ is multiplicity free. 

In the next section, we will study some examples of $E_\tau$ such that $\mathbb{D}(E_\tau)$ are commutative and 
the homomorphisms (\ref{eqn:ghch}) are surjective. 

%
%


\subsection{Spherical functions}

Let $(\tau,V_\tau)$ be an irreducible representation of $K$. We call a function 
$f\,:\,G\rightarrow \text{End}\,(V_\tau)$ is $\tau$-spherical if 
it satisfies the condition
\[
f(k_1xk_2)=\tau(k_2)^{-1}f(x)\tau(k_1)^{-1}\quad (x\in G,\,\,k_1,\,k_2\in K).
\]
Alternatively, it is naturally identified with a function 
$f\,:\,G\rightarrow V_\tau^*\otimes V_\tau$ that satisfies
\[
f(k_1xk_2)=\tau^*(k_1)\otimes\tau(k_2)^{-1} f(x)\quad (x\in G,\,\,k_1,\,k_2\in K).
\]
By the Cartan decomposition $G=KAK$, a $\tau$-spherical function $f$ 
is determined by its restriction to $A$. 
For a differential operator $D$ on $E_\tau$ or an element of $U(\g_\C)$, 
there exists a differential operator $R_\tau(D)$ 
on $C^\infty(A,\,(V_\tau^*\otimes V_\tau)^M)$ 
that satisfies 
\[
D f|_A=R_\tau(D) (f|_A)
\]
for any $\tau$-spherical functions $f$. 
We call $R_\tau(D)$ the $\tau$-radial part of $D$. 
We recall two well-known lemmas. Fix an element 
\[
H\in\mathfrak{a}_+=\{X\in \mathfrak{a}\,:\,
\alpha(X)>0\quad (\alpha\in\Sigma^+)\}
\]
 and put $a=\exp H\in A_+=\exp\mathfrak{a}_+$.

\begin{lemma}[\cite{War} Proof of Proposition 9.1.2.11, \cite{Knapp} Lemma 8.24]
\label{lem:radial1}
For $\alpha\in\Sigma^+$, $X\in\g^\alpha$, 
we have 
\[
X-\theta X=\coth \alpha(H)\,(X+\theta X)-\frac{1}{\sinh\alpha(H)}
\text{\rm Ad}(a^{-1})(X+\theta X).
\]
\end{lemma}


Let $\Omega,\,\Omega_\mathfrak{a}$, and $\Omega_\mathfrak{m}$ denote 
the Casimir elements in $U(\g_\C),\,U(\mathfrak{a}_\C)$, and $U(\mathfrak{m}_\C)$, 
respectively. For $\alpha\in \Sigma$, let $H_\alpha$ be the element 
of $\mathfrak{a}$ such that $\alpha(X)=B(H_\alpha,X)$ for all $X\in\mathfrak{a}$, 
where $B$ is the Killing form for $\g$. 
For each $\alpha\in \Sigma^+$ we choose a basis 
$\{X_{\alpha,i}\}_{1\leq i\leq m_\alpha}$ of $\g^\alpha_\C$ that is orthonormal with respect 
to the inner product $(X,\,Y)=-B(X,\,\theta Y)$. 
We write $X_{\alpha,i}=Z_{\alpha,i}+Y_{\alpha,i}$ where $Z_{\alpha,i}\in\mathfrak{k}_\C$ 
and $Y_{\alpha,i}\in\mathfrak{p}_\C$. 

\begin{lemma}[\cite{War} Proposition 9.1.2.11]
\label{lemma:radcas}
\begin{align*}
\Omega= & \Omega_{\mathfrak{a}}+\Omega_{\mathfrak{m}}+
\sum_{\alpha\in \Sigma^+}m_\alpha\coth\alpha(H)\,H_{\alpha} \\
& +2\sum_{\alpha\in \Sigma^+}\sum_{i=1}^{m_\alpha}
\sinh^{-2}\alpha(H)\,
\{\text{\emph{Ad}}(a^{-1})(Z_{\alpha,i} ^2)+Z_{\alpha,i} ^2
\\
&  \phantom{bbaaabbabbbbbb}
-2\cosh\alpha(H) \,\text{\emph{Ad}}(a^{-1})(Z_{\alpha,i})\,Z_{\alpha,i}\}.
\end{align*}
\end{lemma}

\section{
Invariant differential operators on a homogeneous vector bundle over a 
symmetric space of type $A_2$
}

In this section, we consider symmetric spaces 
\[
G/K=SL(3,\mathbb{K})/SU(3,\mathbb{K})\,\,
(\mathbb{K}=\mathbb{R},\,\mathbb{C},\,\mathbb{H}),
\]
that is, 
\[
G/K=SL(3,\mathbb{R})/SO(3),\,\,SL(3,\mathbb{C})/SU(3),\,\,SU^*(6)/Sp(3).
\]
 The restricted root system $\Sigma$ of $G/K$ 
is  of type $A_2$ and the Weyl group $W$ is isomorphic to $S_3$. 
We regard $\aaa$ as a subspace of $\R^3$ 
\[
\aaa\simeq \left\{
(t_1,\, t_2,\,t_3)\,:\,t_1,\,t_2,\,t_3\in \R,\,t_1+t_2+t_3=0\right\}.
\]
We will give the above identification explicitly in each case of $\mathbb{K}=\mathbb{R},\,\mathbb{C},\,\mathbb{H}$ 
in the following subsections. 
We put 
\begin{equation}\label{eqn:not_tij}
\partial_i=\displaystyle\frac{\partial}{\partial t_i},\,\,
\partial_i'=\partial_i-\frac13(\partial_1+\partial_2+\partial_3),\,\,
t_{ij}=t_i-t_j\,\,(1\leq i\not=j\leq 3).
\end{equation}
In the $GL$-picture, functions on $\mathfrak{a}$ are regarded as functions on $\R^3$ that are killed by 
$\partial_1+\partial_2+\partial_3$. 

Let $\tau$ be the standard representation of $K$ and $E_\tau\rightarrow G/K$ be the 
associated homogeneous vector bundle. 
We have the following theorem for the algebra $\mathbb{D}(E_\tau)$ of invariant differential operators on $E_\tau$. 

\begin{theorem}\label{thm:sym} Let $G/K=SL(3,\mathbb{K})/SU(3,\mathbb{K})\,\,
(\mathbb{K}=\mathbb{R},\,\mathbb{C},\,\mathbb{H})$ and $\tau$ be the standard representation of $K$. 

\noindent
\emph{(i)} $\tau|_M$ decomposes into multiplicity free sum of three irreducible representations of $M$. 
These $M$-modules are in a single $W$-orbits. 

\noindent
\emph{(ii)}  $\mathbb{D}(E_\tau)$ is commutative. 
Moreover, $\gamma_\tau$ in \emph{(\ref{eqn:ghch})} is surjective and gives an algebra isomorphism. 
The composition of $\gamma_\tau$ and the projection 
to the third factor gives an isomorphism $\mathbb{D}(E_\tau)\simeq U(\mathfrak{a}_\C)^{W_{e_1-e_2}}$, 
where $W_{e_1-e_2}=\{e,\,s_{12}\}$. 
Moreover, $\mathbb{D}(E_\tau)$ is generated by two operators $D_1$ and $D_2$, which are algebraically independent 
and the order of $D_i$ is $i$ for $i=1,\,2$. $D_1$ is unique up to a constant multiple and a 
constant difference.   

\noindent
\emph{(iii)}
Let $k=1/2,\,1,\,2$ for $\mathbb{K}=\mathbb{R},\,\mathbb{C},\,\mathbb{H}$, respectively. 
We may take 
$D_2$ to be the image of a constant multiple of $\Omega-\tau(\Omega_\mathfrak{m})$. 
With respect to a basis of $(V_{\tau^*}\otimes V_\tau)^M$, $\tau$-radial parts of $D_1$ and $D_2$ 
have matrix expressions
\begin{align*}
R_\tau(D_1) & =  
{\begin{pmatrix}
\partial_1' & 0 & 0 \\
0 & \partial_2' & 0 \\
0 & 0 & \partial_3'
\end{pmatrix}}  \\
& \phantom{a}+{k}\begin{pmatrix}
{\scriptstyle\coth t_{12}+\coth  t_{13}} & -\frac{1}{\sinh t_{12}} & -\frac{1}{\sinh t_{13}}\\
\frac{1}{\sinh t_{12}} & {\scriptstyle-\coth t_{12}+\coth t_{23}}& -\frac{1}{\sinh t_{23}}\\
\frac{1}{\sinh t_{13}} &\frac{1}{\sinh t_{23}} & {\scriptstyle-\coth t_{13}-\coth t_{23}}
\end{pmatrix} 
, \\
R_\tau(D_2) & = {L_2}
\\
& \phantom{a}+k  \begin{pmatrix}
\frac{1}{\sinh^2 t_{12}}+\frac{1}{\sinh^2 t_{13}} & -\frac{\cosh t_{12}}{\sinh^2 t_{12}} & -\frac{\cosh t_{13}}{\sinh^2 t_{13}} \\
 -\frac{\cosh t_{12}}{\sinh^2 t_{12}} & \frac{1}{\sinh^2 t_{12}}+\frac{1}{\sinh^2 t_{23}} & -\frac{\cosh t_{23}}{\sinh^2 t_{23}} \\
 -\frac{\cosh t_{13}}{\sinh^2 t_{13}} &  -\frac{\cosh t_{23}}{\sinh^2 t_{23}} & \frac{1}{\sinh^2 t_{13}}+\frac{1}{\sinh^2 t_{23}} 
\end{pmatrix}  , 
\end{align*}
where 
\[
L_2= \partial_1'\partial_2'+\partial_2'\partial_3'+\partial_3'\partial_1'
-k\sum_{1\leq i<j\leq 3} (\coth t_{ij})(\partial_i-\partial_j) .
\]
\end{theorem}

We will prove the above theorem in the following subsections. 

\medskip
By the identification $U(\mathfrak{a}_\C)\simeq S(\mathfrak{a}_\C)$, 
we have 
\begin{align*}
& \gamma_\tau(D_1)(\lambda)=\text{diag}(\lambda_1,\lambda_2,\lambda_3), \\
& \gamma_\tau(D_2)(\lambda)=
\lambda_1\lambda_2+\lambda_2\lambda_3+\lambda_3\lambda_1+4k^2
\end{align*}
for $\lambda\in\mathfrak{a}_\C^*$. 
Hence, 
the operator $D_1$ and $D_2$ in Theorem~\ref{thm:sym} correspond to 
\[
\lambda_3\quad  \text{and}\quad \lambda_1\lambda_2+\lambda_2\lambda_3+\lambda_3\lambda_1+4k^2
\]
respectively 
by the isomorphism 
\[
\mathbb{D}(E_\tau)\simeq U(\mathfrak{a}_\C)^{W_{e_1-e_2}}
\simeq 
S(\mathfrak{a}_\C)^{W_{e_1-e_2}}.
\]

There are several choices of a second order operator such that $D_1$ and it generate 
$\mathbb{D}(E_\tau)$. For example, there exists an invariant differential operator $\tilde{D}_2$ such that 
its $\tau$-radial part is of the form
\[
R_\tau(\tilde{D}_2)=
\begin{pmatrix}
\partial_2'\partial_3' & 0 & 0 \\
0 & \partial_3'\partial_1' & 0 \\
0 & 0 & \partial_1'\partial_2'
\end{pmatrix} +\text{(lower order terms)}.
\]

\begin{remark}
The operator $L_2$ in the above theorem is a constant multiple of the radial part of 
the Laplace-Beltrami operator on $G/K$ (\cite[Proposition 3.9]{Hel}). 
\end{remark}

\begin{remark}
Though the restricted root system of the exceptional 
symmetric pair $(\mathfrak{e}_{6(-26)}, \,\mathfrak{f}_4)$ 
 is  also of type $A_2$, it seems that there are no $K$-type $\tau$ such that  
results like Theorem~\ref{thm:sym} hold. 
\end{remark}

\begin{remark}
For $G/K=SL(2,\mathbb{K})/SU(2,\mathbb{K})\,\,(\mathbb{K}=\mathbb{C},\,\mathbb{H})$, that is 
$G/K=SL(2,\mathbb{C})/SU(2),\,SU^*(4)/Sp(2)$, the algebra $\mathbb{D}(E_\tau)$ is commutative 
for the standard representation $\tau$ of $K$. 
$SL(2,\mathbb{C})$ and $SU^*(4)$ are isomorphic to $Spin(3,1)$ and $Spin(5,1)$, respectively. 
Moreover the standard representations of $SU(2)$ and $Sp(2)$ correspond to the spin representations 
of $Spin(3)$ and $Spin(5)$, respectively. 
For the symmetric space $G/K=Spin(2m+1,1)/Spin(2m+1)\,\,(m\geq 1)$ and 
the spin representation $\tau$ of $K$, the algebra $\mathbb{D}(E_\tau)$ 
  is commutative and generated by the Dirac operator (\cite{G}, \cite{C})．

\end{remark}

\subsection{The case of $SL(3,\R)/SO(3)$}\label{subsec:R}
\subsubsection{Notation}\label{subsec:not1}
Let $G=SL(3,\R)$ and $K=SO(3)$. The Lie algebras of $G$ and $K$ are
\[
\g=\mathfrak{sl}(n,\R)=\{X\in M(3,\R)\,:\,\text{Tr} \,X=0\}
\]
 and $\kk=\mathfrak{so}(3)$. The Cartan involution of $\g$ is given by $\theta X=-{}^t X$ for $X\in \g$. 
We have the Cartan decomposition $\g=\kk+\p$ into $\pm 1$-eigenspaces of $\theta$. 
Here $\kk$ and $\p$ consist of the real 
skew-symmetric and symmetric matrices respectively. 

Let $\{K_1,\,K_2,\,K_3\}$ denote the basis of $\kk$ defined by
\[
K_1=\begin{pmatrix} & & \\ & & -1 \\ & 1 &\end{pmatrix},\quad
K_2=\begin{pmatrix} & & 1 \\ & &  \\ -1&  &\end{pmatrix},\quad
K_3=\begin{pmatrix} & -1 & \\ 1 & &  \\ &  &\end{pmatrix}.
\]

Let $\aaa$ denote the subspace of $\p$ consisting of diagonal matrices 
\begin{equation}\label{eqn:a1}
\aaa=\left\{
\begin{pmatrix}t_1 & 0 & 0 \\ 0 & t_2 & 0 \\ 0 & 0 & t_3\end{pmatrix}\,:\,t_1,\,t_2,\,t_3\in \R,\,t_1+t_2+t_3=0\right\}.
\end{equation}

The centralizer $M$ of $\aaa$ in $K$ is given by
\[
M=\left\{\begin{pmatrix}1 & & \\ & -1 & \\ & & -1\end{pmatrix},\,\begin{pmatrix} -1 & & \\ & 1 & \\ &  & -1\end{pmatrix},\,
\begin{pmatrix}-1 & & \\ & -1 & \\ & & 1\end{pmatrix},\,\begin{pmatrix}1 & & \\ & 1 & \\ & & 1\end{pmatrix}\right\}.
\]


\subsubsection{Invariant differential operators}
Let $\tau$ denote the standard representation of $K$ on $V_\tau=\C^3$. Equip $\C^3$ with the standard basis 
$\{\boldsymbol{e}_1={}^t(1,0,0),\,\boldsymbol{e}_2={}^t(0,1,0),\,\boldsymbol{e}_3={}^t(0,0,1)\}$. 

Since  $\tau|_M$ decomposes into multiplicity free sum of irreducible representations,
  the algebra $\mathbb{D}(E_\tau)$ 
of invariant differential operators 
on the homogeneous vector bundle $E_\tau\rightarrow G/K$ associated with $\tau$ is
commutative (\cite{D}). Moreover, the irreducible constituents of $\tau|_M$ are in a single 
$W$-orbits

\subsubsection{First order invariant differential operator}
$(\tau,\C^3)$ and $(\text{Ad},\p_\C)$ are irreducible representation of dimension $3$ and $5$ respectively 
and the tensor product $\p_\C\otimes\C^3$ contains $(\tau,\C^3)$ with multiplicity 1. 
Hence, there is a unique first order operator in $\mathbb{D}(E_\tau)$ up to a constant multiple. 

$\tau$-radial part of a first order invariant differential operator was computed by Sono \cite[Theorem 6.4]{Sono}. 
We will give a proof in a way parallel to our 
 discussions for the cases of 
$SL(3,\mathbb{K})/SU(3,\mathbb{K})\,\,(\mathbb{K}=\mathbb{C},\,\mathbb{H})$, which are given in subsequent 
sections. 

Let $E_{ij}$ denote the $3\times 3$ matrix with $(i,j)$ entry $1$ and the other entries 0 and put $E_i'=
E_{ii}-\frac13(E_{11}+E_{22}+E_{33})$. Then 
\[
\{E_{ij}\,(1\leq i\not=j\leq 3),\,E_k'\,(1\leq k\leq 3)\}
\]
 forms a basis of 
$\mathfrak{sl}(3,\R)$. 

\begin{lemma}[Manabe-Ishii-Oda \cite{MIO}]\label{lemma:firsta}
The dimension of $(\p_\C\otimes\text{\emph{End}}(V_\tau))^K$ is $1$ and a basis vector 
is represented by the following matrix
\begin{equation}\label{eqn:invdo1}
\begin{pmatrix}
E_{1}' & \frac12(E_{12}+E_{21}) & \frac12(E_{13}+E_{31}) \\
\frac12(E_{12}+E_{21}) & E_{2}' &  \frac12(E_{23}+E_{32}) \\
 \frac12(E_{13}+E_{31}) &  \frac12(E_{23}+E_{32}) & E_{3}'
\end{pmatrix}
\end{equation}
with respect to the basis $\{\boldsymbol{e}_1,\boldsymbol{e}_2,\boldsymbol{e}_3\}$ of $V_\tau$. 
\end{lemma}
\begin{proof}
A first order invariant differential operator is given by 
\cite[Lemma 3.3]{MIO} with respect to the standard weight vectors of $\mathfrak{so}(3)\simeq \mathfrak{su}(2)$, 
Applying the change of basis  given in \cite[\S 4.3]{MIO}, the matrix expression follows.  
\end{proof}

\begin{remark}
There exists an element $D\in U(\g)^K$ such that (\ref{eqn:invdo1}) is the image of $D$ under the surjective map 
 $U(\g_\C)^K\rightarrow (U(\g_\C)\!\otimes _{\kk_\C}\! \text{End}(V_\tau))^K$. 
We may take $D$ in $(\p_\C\otimes \kk_\C\otimes \kk_\C)^K$. 
The proof of \cite[Lemma 1]{D} with some additional computations shows that $D$ is 
\begin{align}
& E_1'  K_1^2  +E_2'K_2^2+E_3'K_3^2+ \frac12(E_{12}+E_{21})(K_1K_2+K_2K_1) \label{eqn:invdo1b} \\
& +\frac12(E_{13}+E_{31})(K_1K_3+K_3K_1)+\frac12(E_{23}+E_{32})(K_2K_3+K_3K_2). \notag
\end{align}
\end{remark}

We will give the $\tau$-radial part of (\ref{eqn:invdo1}). 
The space $(V_{\tau^*}\otimes V_{\tau})^M$ is a 3-dimensional vector space with 
a basis 
\begin{equation}\label{eqn:basisr}
\{\boldsymbol{e}_1^*\otimes \boldsymbol{e}_1,\,\boldsymbol{e}_2^*\otimes \boldsymbol{e}_2,\,
\boldsymbol{e}_3^*\otimes \boldsymbol{e}_3\}.
\end{equation}
The following proposition gives the radial part of (\ref{eqn:invdo1}) with respect to this basis in the coordinates (\ref{eqn:a1}) of $\aaa$.

\begin{proposition}\label{prop:radial1}\emph{(Sono \cite[Theorem 6.4]{Sono})}  
With respect to the basis \emph{(\ref{eqn:basisr})} of $(V_{\tau^*}\otimes V_\tau)^M$,  
the $\tau$-radial part of 
\emph{(\ref{eqn:invdo1})} 
has the following matrix expression
\[
\begin{pmatrix}
\partial_1'+\frac{\coth t_{12}+\coth t_{13}}{2} & -\frac{1}{2\sinh t_{12}} & -\frac{1}{2\sinh t_{13}}\\
\frac{1}{2\sinh t_{12}} & \partial_2'+\frac{\coth t_{21}+\coth t_{23}}{2}& -\frac{1}{2\sinh t_{23}} \\
\frac{1}{2\sinh t_{13}} & \frac{1}{2\sinh t_{23}} &\partial_3'+\frac{\coth t_{31} +\coth t_{32}}{2}
\end{pmatrix}.
\]
\end{proposition}
\begin{proof}
We consider the first row of the matrix (\ref{eqn:invdo1}). 
By Lemma \ref{lem:radial1}, we have
\begin{align*}
E_1' & \boldsymbol{e}_1+\frac12(E_{12}+E_{21})\boldsymbol{e}_2+\frac12(E_{13}+E_{31}) \boldsymbol{e}_3 \\
 = & E_1' \boldsymbol{e}_1+\left(-\frac12\coth t_{12}\, K_3+\frac{1}{2\sinh t_{12}}\text{Ad}(a^{-1})K_3\right)\boldsymbol{e}_2 
\\
 & +
\left(\frac12\coth t_{13}\, K_2
-\frac{1}{2\sinh t_{13}}\text{Ad}(a^{-1})K_2\right)\boldsymbol{e}_3 
\\
 = & \left(E_1'+\frac12\coth t_{12}+\frac12\coth t_{13}\right) \boldsymbol{e}_1
+\frac{1}{2\sinh t_{12}}\text{Ad}(a^{-1})K_3 \boldsymbol{e}_2  \\
& -\frac{1}{2\sinh t_{13}}\text{Ad}(a^{-1})K_2 \boldsymbol{e}_3. 
\end{align*}
Since $K_2\boldsymbol{e}_1^*=\boldsymbol{e}_3^*$ and $K_3\boldsymbol{e}_1^*=-\boldsymbol{e}_2^*$, 
 the action  of the first row of  (\ref{eqn:invdo1}) on the coefficient of $\boldsymbol{e}_1^*\otimes\boldsymbol{e}_1$ is 
given by 
\begin{align*}
&  \left(E_1'+\frac12\coth t_{12}+\frac12\coth t_{13}\right)  \boldsymbol{e}_1^*\otimes \boldsymbol{e}_1
+\frac{1}{2\sinh t_{12}} \boldsymbol{e}_2^*\otimes \boldsymbol{e}_2 \\
& \phantom{aaaaa}+\frac{1}{2\sinh t_{13}} \boldsymbol{e}_3^*\otimes \boldsymbol{e}_3. 
\end{align*}
Similarly, we can compute the radial part of the second and third rows of the matrix (\ref{eqn:invdo1}) and 
obtain the matrix expression.
\end{proof}

\subsubsection{Radial part of the Casimir operator}

The following proposition gives the $\tau$-radial part of 
 the Casimir operator $\Omega$.

\begin{proposition}\label{prop:casimir1}
The $\tau$-radial part $R_\tau(\Omega)$ of the Casimir operator is given by 
\begin{align}
-3 & R_\tau  (\Omega)= 
\partial_1\partial_2+\partial_2\partial_3+\partial_3\partial_1 \label{eqn:casimirrad1a} \\
& -\frac{1}{2}(\coth t_{12}(\partial_1-\partial_2)+\coth t_{13}(\partial_1-\partial_3)+\coth t_{23}(\partial_2-\partial_3)) \notag \\
&+\frac12\begin{pmatrix}
\frac{1}{\sinh^2 t_{12}}+\frac{1}{\sinh^2 t_{13}} & -\frac{\cosh t_{12}}{\sinh^2 t_{12}} & -\frac{\cosh t_{13}}{\sinh^2 t_{13}} \\
 -\frac{\cosh t_{12}}{\sinh^2 t_{12}} & \frac{1}{\sinh^2 t_{12}}+\frac{1}{\sinh^2 t_{23}} & -\frac{\cosh t_{23}}{\sinh^2 t_{23}} \\
 -\frac{\cosh t_{13}}{\sinh^2 t_{13}} &  -\frac{\cosh t_{23}}{\sinh^2 t_{23}} &\frac{1}{\sinh^2 t_{13}}+\frac{1}{\sinh^2 t_{23}} \notag
\end{pmatrix}
\end{align}
 with respect to the basis \emph{(\ref{eqn:basisr})} of $(V_{\tau^*}\otimes V_\tau)^M$.

\end{proposition}

\begin{proof}
The Killing form of $\g=\mathfrak{sl}(3,\R)$ is given by 
$B(X,Y)=6\,\text{Tr}\,XY$ for $X,\,Y\in \g$. 
$\{(E_1'-E_2')/\sqrt{12},\,E_3'/2\}$ forms an orthonormal basis of $\aaa$. 
Moreover,  $E_{ij}\in \g^{e_i-e_j}$ and 
$B(E_{ij}/\sqrt{6},E_{ji}/\sqrt{6})=1$ for $1\leq i\not=j\leq 3$. The Casimir operator 
is given by 
\begin{align}
\Omega=&\frac{1}{12}(E_1'-E_2')^2 +\frac14E_3'^2 \label{eqn:casimir1a}
\\
& +\frac{1}{6}(E_{12}E_{21}+E_{21}E_{12}+E_{13}E_{31}+E_{31}E_{13}+E_{23}E_{32}+E_{32}E_{23}) \notag
\\
=& \frac19(E_1+E_2+E_3)^2-\frac13(E_1E_2+E_2E_3+E_3E_1) \notag \\
& +\frac{1}{6}(E_{12}E_{21}+E_{21}E_{12}+E_{13}E_{31}+E_{31}E_{13}+E_{23}E_{32}+E_{32}E_{23}) . \notag
\end{align}
By 
Lemma~\ref{lemma:radcas}, we have
\begin{align}
\Omega=& \frac19(\partial_1+\partial_2+\partial_3)^2-\frac13(\partial_1\partial_2+\partial_2\partial_3+\partial_3\partial_1) \notag \\
& +\frac{1}{6}(\coth t_{12}(\partial_1-\partial_2)+(\coth t_{13}(\partial_1-\partial_3)+(\coth t_{23}(\partial_2-\partial_3)) \notag \\
& -\frac{1}{12} \sinh^{-2}t_{12}\{-\text{Ad}(a^{-1})(K_3^2)-K_3^2+2\cosh t_{12}\,(\text{Ad}(a^{-1})K_3)K_3\} \notag \\
&  -\frac{1}{12}  \sinh^{-2}t_{13}\{-\text{Ad}(a^{-1})(K_2^2)-K_2^2+2\cosh t_{13}\,(\text{Ad}(a^{-1})K_2)K_2\} \notag \\
& -\frac{1}{12}  \sinh^{-2}t_{23}\{-\text{Ad}(a^{-1})(K_1^2)-K_1^2+2\cosh t_{23}\,(\text{Ad}(a^{-1})K_1)K_1\}. \notag
\end{align}

Since $K_i \boldsymbol{e}_i=0,\,K_i\boldsymbol{e}_j=\boldsymbol{e}_k,\,,K_i\boldsymbol{e}_k=-\boldsymbol{e}_j$ for $(i,j,k)=(1,2,3),\,(2,3,1)$ and $(3,1,2)$, and 
the actions on $\boldsymbol{e}_i^*$'s have opposite signs, the action of $-3\Omega$ on the coefficient of $\boldsymbol{e}_1^*\otimes \boldsymbol{e}_1$ is 
\begin{align*}
& \left(-\frac13(\partial_1+\partial_2+\partial_3)^2+
\partial_1\partial_2+\partial_2\partial_3+\partial_3\partial_1\right) \boldsymbol{e}_1^*\otimes \boldsymbol{e}_1 \notag \\
& -\frac12\sum_{1\leq i<j\leq 3}\coth t_{ij}(\partial_i-\partial_j) \boldsymbol{e}_1^*\otimes \boldsymbol{e}_1 
 +\frac12\left(\frac{1}{\sinh ^2 t_{12}}+\frac{1}{\sinh^2 t_{13}}\right)\boldsymbol{e}_1^*\otimes \boldsymbol{e}_1 \\
& 
-\frac{\cosh t_{12}}{2\sinh^2 t_{12}}\boldsymbol{e}_2^*\otimes \boldsymbol{e}_2
-\frac{\cosh t_{13}}{2\sinh^2 t_{13}}\boldsymbol{e}_3^*\otimes \boldsymbol{e}_3.
\end{align*}
The actions of $\Omega$ on $\boldsymbol{e}_2\otimes \boldsymbol{e}_2^*$ and $\boldsymbol{e}_3\otimes \boldsymbol{e}_3^*$ are given in 
similar way and $\partial_1+\partial_2+\partial_3$ acts by zero on functions on  $\mathfrak{a}$.  Thus 
we have the matrix expression (\ref{eqn:casimirrad1a}). 
\end{proof}


\subsection{The case of $SL(3,\C)/SU(3)$}\label{subsec:C}
\subsubsection{Notation}\label{subsec:not2}
Throughout this section we use notation of \cite{HOd}. Namely, the imaginary unit is denoted by $J$. 
Hence complex numbers are of the form $\alpha+J\beta$ with $\alpha,\,\beta\in\R$. 
For a real vector space $\mathfrak{l}$, its complexification $\mathfrak{l}\otimes_\R \C$ is denoted by 
$\mathfrak{l}_\C$. Let $\sqrt{-1}$ denote the complex structure on $\mathfrak{l}_\C$. 
Namely, 
\[
\mathfrak{l}_\C=\{X+\sqrt{-1}Y\,:\,X,\,Y\in \mathfrak{l}\}, 
\]
and $\sqrt{-1}$ acts on $\mathfrak{l}_\C$ as a linear operator with $(\sqrt{-1})^2=-\text{id}_{\mathfrak{l}_\C}
$, and
\[
 (\alpha+J\beta) Z=\alpha Z+\beta \sqrt{-1}Z\quad (\alpha,\,\beta\in\R,\,Z\in \mathfrak{l}).
\]

We employ the $GL$-picture as in \cite{HOd}. 
Let $G=GL(3,\C)$ and $K=U(3)$. The Lie algebras of $G$ and $K$ are
\begin{align*}
& \g=\mathfrak{gl}(3,\C)=M(3,\C)=\{3\times 3\text{ complex matrices}\}, \\
& \kk=\mathfrak{u}(3)=\{X\in M(3,\C)\,:\,X^*=-X\}
\end{align*}
respectively. Here $X^*={}^t \overline{X}$ for $X\in \mathfrak{gl}(3,\C)$. Define $\theta X=-X^*$ for 
$X\in \g$. Then $\theta$ is the Cartan involution of $\g$ that satisfies 
\[
\kk=\g^\theta
=\{X\in\g\,:\,\theta X=X\}.
\]
Let $\p$ denote the $-1$-eigenspace of $\theta$ in $\g$:
\[
\p=\{X\in M(3,\C)\,:\,X^*=X\}.
\]
We have a Cartan decomposition $\g=\kk\oplus \p$. Let $\aaa$ denote the set of the diagonal matrices in 
$\p$:
\begin{equation}\label{eqn:a2}
\aaa=\{\text{diag}\,(t_1,t_2,t_3)\,:\,t_i\in\R\,\,(1\leq i\leq 3)\}.
\end{equation}
Then $\aaa$ is a maximal abelian subspace of $\p$. 
Let $M$ denote the centralizer of $\aaa$ in $K$. We have
\[
M=\{\text{diag}\,(u_1,u_2,u_3)\,:\,u_i\in U(1)\,\,(1\leq i\leq 3)\}\simeq U(1)^3.
\]

We view $\g$ and its subalgebras as real Lie algebras. Then the Killing form of 
$\g$ is given by 
\[
B_\R(X,Y)=2\,\text{Re}\,B_{\g}(X,Y)=
12 \,\text{Re}\,(\text{Tr}\,XY-4\,\text{Tr}\,X\,\text{Tr}\,Y)\quad (X,\,Y\in \g). 
\]
The Killing form $B_{\g_\C}$ 
on $\g_\C$ is 
given by linear extension of $B_\R$. 

Let $X,\,Y\in \g$ and $X+\sqrt{-1}Y\in \g_\C$. 
The mapping 
\[
X+\sqrt{-1}Y\mapsto (X+JY)\oplus(X-JY)
\]
 is a Lie algebra isomorphism of 
$\g_\C$ onto $\g\oplus \g$. The inverse mapping is given by 
\[
\g\oplus\g \ni Z\oplus W\mapsto \frac12\{Z+W-\sqrt{-1}J(Z-W)\}\in \g_\C.
\]
Moreover, both the mapping $Z\mapsto \frac12(Z-\sqrt{-1}JZ)$ and 
$W\mapsto \frac12(W+\sqrt{-1}JW)$ are isometries of $\g$ into $\g_\C$ 
with respect to the Killing forms $B_\g$ and $B_{\g_\C}$. 

By the above isomorphism
\begin{equation}
\label{eqn:compisom}
\g_\C\simeq \g\oplus\g,
\end{equation}
 we have an isomorphism 
\begin{equation}
\label{eqn:envisom}
U(\g_\C)\simeq U(\g)\otimes_\C U(\g).
\end{equation}

For $1\leq i,\,j\leq 3$, 
let $E_{ij}$ (resp. $E_{ij}'$) denote the $3\times 3$ matrix with $(i,j)$-entry $1$ 
(resp. $J$) and the remaining entries $0$. 
Then $\{E_{ij},\,
E_{ij}'
\,\,(1\leq i,\,j\leq 3)\}$ forms an $\R$-basis of 
$\g$. 
Define $H_{ij}\in\aaa,\,H_{ij}'\in \mathfrak{k}$ by $H_{ij}=E_{ii}-E_{jj},\,H_{ij}'=E_{ii}'-E_{jj}'$. 
Let $I_3$ denote the identity matrix in $M(3,\C)$ and put $I_3'=JI_3\in \kk$. 

The semisimple part $\g_0$ of $\g$ is 
\[
\g_0=\mathfrak{sl}(3,\C)=\{X\in M(3,\C)\,:\,\text{Tr}\,X=0\}.
\]
Put $\kk_0=\kk\cap \g_0=\mathfrak{su}(3)$, $\p_0=\p\cap \g_0$, and $\aaa_0=\aaa\cap\g_0$. 
Then $\kk=\kk_0\oplus \R I_3'$, $\p=\p_0\oplus \R I_3$, and 
$\{H_{ij},\,H_{ij}'\,\,(1\leq i<j\leq 3)\}$ forms a basis of $\mathfrak{a}\cap \mathfrak{sl}(3,\C)$. 
Let $G_0=SL(3,\C)$, $K_0=SU(3)$, and $M_0=M\cap G_0$. 

Define $I_3^\kk,\,H_{ij}^\kk,\,E_{ij}^\kk\in \kk_\C$ and $I_3^\p,\,H_{ij}^\p,\,E_{ij}^\p\in \p_\C$ by
\begin{align*}
& I_3^\kk=\sqrt{-1}I_3',\,\, H_{ij}^\kk=\sqrt{-1}JH_{ij},\,\,
E_{ij}^\kk=\frac12\left\{(E_{ij}-E_{ji})-\sqrt{-1}(E_{ij}'+E_{ji}')\right\} \\
& I_3^\p=I_3,\,\,
H_{ij}^\p=H_{ij},\,\, E_{ij}^\p=\frac12\left\{(E_{ij}+E_{ji})-\sqrt{-1}(E_{ij}'-E_{ji}')\right\}.
\end{align*}
The element $E_{ij}\oplus 0$ and $0\oplus E_{ij}$ in $\g\oplus\g$ correspond 
$\frac12(E_{ij}^\p+E_{ij}^\kk)$ and $\frac12(E_{ij}^\p-E_{ij}^\kk)$ in $\g_\C$ under the 
isomorphism (\ref{eqn:compisom}), respectively.  Similar identifications hold for $I_3$ and $H_{ij}$. 
Define $\E_i^\p\in \p_0$ ($1\leq i\leq 3$) by
\[
\E_i^\p=E_{ii}-\frac13 I_3.
\]

Let $\mathfrak{t}$ denote the set of the diagonal matrices in $\kk$:
\[
\mathfrak{t}=\{\text{diag}\,(Jt_1,Jt_2,Jt_3)\,:\,t_i\in\R\,\, (1\leq i\leq 3)\}.
\]
Then $\mathfrak{t}$ is a Cartan subalgebra of $\kk$. 
Let $\varepsilon_i$ denote the linear form on $\mathfrak{t}$ defined by 
$\varepsilon_i(\text{diag}\,(Jt_1,Jt_2,J_3))=t_i\,\,(1\leq i\leq 3)$. 
The set of the dominant 
integral weights on $\mathfrak{t}$ is given by 
\[
\Lambda=\{\mu=
\mu_1\varepsilon_1+\mu_2\varepsilon_2+\mu_3\varepsilon_3
\,:\,\mu_1\geq \mu_2\geq \mu_3
,\,\mu_i\in\mathbb{Z}\,\,(1\leq i\leq 3)
\}.
\]
The equivalence classes of irreducible representations of $K$ are parametrized by $\Lambda$. 


\subsubsection{Invariant differential operators} 
Let $\tau$ denote the standard representation of $K=U(3)$ on $V_\tau=\C^3$. 
The restriction $\tau|_{K_0}$ is irreducible. 
Equip $\C^3$ with the standard basis 
$\{\boldsymbol{e}_1={}^t(1,0,0),\,\boldsymbol{e}_2={}^t(0,1,0),\,\boldsymbol{e}_3={}^t(0,0,1)\}$. 
The highest weight of $\tau$ is $\varepsilon_1$ and $\boldsymbol{e}_1$ is a highest 
weight vector. 
Since $\tau|_{M_0}$ decomposes into multiplicity-free sum of irreducible 
representations, the algebra $\mathbb{D}(E_\tau)$ of invariant differential operators on 
the homogeneous vector bundle $E_\tau\rightarrow G_0/K_0$ associated 
with $\tau$ is commutative (\cite{D}). Moreover, the irreducible constituents of $\tau|_{M_0}$ 
are in a single $W$-orbits.

\subsubsection{First order invariant differential operator}
The highest weights of  the standard representation $\tau$ of $K_0$ on $V_\tau=\C^3$ and the adjoint representation 
of $K_0$ 
on $(\mathfrak{p}_0)_\C$ are $\varepsilon_1$ and $\varepsilon_1-\varepsilon_3$, respectively. 
By the Littlewood-Richardson rule for $K_0=SU(3)$, $(\mathfrak{p}_0)_\C\otimes V_\tau$ decomposes into 
multiplicity-free sum of three irreducible representations with highest weights 
$\varepsilon_1,\,2\varepsilon_1+2\varepsilon_2,\,3\varepsilon_1+\varepsilon_2$. 
Hence, the isotypic component of $(\mathfrak{p}_0)_\C\otimes V_\tau$ with 
the highest weight $\varepsilon_1$ gives a first order invariant differential operator on $E_\tau$, which is 
unique up to a constant multiple. 

\begin{lemma}\label{lemma:1stc}
The dimension of $((\p_0)_\C\otimes\text{End}(V_\tau))^{K_0}$ is $1$ and a basis vector is given by 
the following matrix
\begin{equation}
\begin{pmatrix}
\E_{1}^\p & E_{12}^\p & E_{13}^\p \\
E_{21}^\p & \E_{2}^\p &  E_{23}^\p \\
E_{31}^\p & E_{32}^\p & \E_{3}^\p
\end{pmatrix}, 
\label{eqn:invdo2}
\end{equation}
with respect to the basis $\{\boldsymbol{e}_1,\boldsymbol{e}_2,\boldsymbol{e}_3\}$. 
\end{lemma}
\begin{proof}
We employ the $GL$-picture as in \cite{HOd}. 
The highest weight of 
$\tau$ is $\varepsilon_1$ and the Gelfand-Zelevinsky basis of $V_\tau$ is parametrized 
by the set of G-patterns
\[
\begin{pmatrix} 1 \,\,0\,\,0 \\ 1 \,\,0 \\ 1\end{pmatrix},\quad 
\begin{pmatrix} 1 \,\,0\,\,0 \\ 1 \,\,0 \\ 0\end{pmatrix},\quad
\begin{pmatrix} 1 \,\,0\,\,0 \\ 0 \,\,0 \\ 0\end{pmatrix},
\]
which correspond to $\boldsymbol{e}_1,\,\boldsymbol{e}_2,\,\boldsymbol{e}_3$ respectively 
(cf. \cite[Lemma 4.1]{HOd}). 

There is a unique $K$-homomorphism $\iota_2$ of 
$V_\tau$ into $(\p_0)_\C\otimes V_\tau$ and an explicit description of $\iota_2$ is given by 
\cite[Lemma 4.3, Theorem 4.4]{HOd}. For a G-pattern $M$, let $f(M)$ denote the corresponding 
Gelfand-Zelevinsly basis. We put
\[
u_1=f\begin{pmatrix} 1 \,\,0\,\,0 \\ 1 \,\,0 \\ 1\end{pmatrix},\quad 
u_2=f\begin{pmatrix} 1 \,\,0\,\,0 \\ 1 \,\,0 \\ 0\end{pmatrix},\quad
u_3=f\begin{pmatrix} 1 \,\,0\,\,0 \\ 0 \,\,0 \\ 0\end{pmatrix}.
\]
Applying \cite[Theorem 4.4]{HOd} to the highest weight $\varepsilon_1$ of $\tau$, we have 
\begin{align*}
& \iota_2(u_1)=
f\begin{pmatrix} 1 \,\,0\,\,-\!1 \\ 1 \,\,-\!1 \\ 0\end{pmatrix}\otimes u_1-
f\begin{pmatrix} 1 \,\,0\,\,-\!1 \\ 1 \,\,-\!1 \\ 1\end{pmatrix}\otimes u_2+
f\begin{pmatrix} 1 \,\,0\,\,-\!1 \\ 1 \,\,0 \\ 1\end{pmatrix}\otimes u_3 \\
& \iota_2(u_2)=
f\begin{pmatrix} 1 \,\,0\,\,-\!1 \\ 1 \,\,-\!1 \\ -\!1\end{pmatrix}\otimes u_1-
\left(
f\begin{pmatrix} 1 \,\,0\,\,-\!1 \\ 1 \,\,-\!1 \\ 0\end{pmatrix}+
f\begin{pmatrix} 1 \,\,0\,\,-\!1 \\ 0 \,\,0 \\ 0\end{pmatrix}
\right)\otimes u_2 \\
& \phantom{aaaaaaaaaaaaaaaaaaaaaaaaaaaaaaaaaaaaaaa}+
f\begin{pmatrix} 1 \,\,0\,\,-\!1 \\ 1 \,\,0 \\ 0\end{pmatrix}\otimes u_3
\\
& \iota_2(u_3)=
f\begin{pmatrix} 1 \,\,0\,\,-\!1 \\ 0 \,\,-\!1 \\ -\!1\end{pmatrix}\otimes u_1-
f\begin{pmatrix} 1 \,\,0\,\,-\!1 \\ 0 \,\,-\!1 \\ 0\end{pmatrix}\otimes u_2+
f\begin{pmatrix} 1 \,\,0\,\,-\!1 \\ 0 \,\,0 \\ 0\end{pmatrix}\otimes u_3.
\end{align*}
By \cite[Lemma 4.3]{HOd}, we have
\[
\begin{pmatrix}\iota_2(u_1)\\ \iota_2(u_2)\\ \iota_2(u_3)
\end{pmatrix}=\begin{pmatrix}
\E_{1}^\p & E_{12}^\p & E_{13}^\p \\
E_{21}^\p & \E_{2}^\p &  E_{23}^\p \\
E_{31}^\p & E_{32}^\p & \E_{3}^\p
\end{pmatrix}\begin{pmatrix}u_1\\ u_2\\ u_3
\end{pmatrix}.
\]
\end{proof}

\begin{remark}
There is a minor misprint in 
 \cite{HOd}. The right hand side of the  fifth line of  \cite[Lemma 4.2]{HOd} should be
\[
-\frac13(H_{12}^\p+2H_{23}^\p).
\]
\end{remark}

We will give the $\tau$-radial part of (\ref{eqn:invdo2}). The contragredient 
representation $\tau^*$ has the highest weight $-\varepsilon_3$. 
The space $(V_{\tau^*}\otimes V_{\tau})^{M_0}$ is a 3-dimensional vector space with 
a basis 
\begin{equation}\label{eqn:basisc2}
\{\boldsymbol{e}_1^*\otimes \boldsymbol{e}_1,\,\boldsymbol{e}_2^*\otimes \boldsymbol{e}_2,\,\boldsymbol{e}_3^*\otimes \boldsymbol{e}_3\}.
\end{equation}
The following proposition gives the radial part of (\ref{eqn:invdo2}) with respect to this basis in the coordinates (\ref{eqn:a2}) of $\aaa$. 

\begin{proposition}\label{prop:radial2}
With respect to the basis \emph{(\ref{eqn:basisc2})} of $(V_{\tau^*}\otimes V_\tau)^{M_0}$, 
the $\tau$-radial part of the first order invariant differential operator \emph{(\ref{eqn:invdo2})} 
has the following matrix expression.
\[
\begin{pmatrix}
{\partial_1'+{\coth t_{12}+\coth t_{13}} }& -\frac{1}{\sinh t_{12}} & -\frac{1}{\sinh t_{13}}\\
\frac{1}{\sinh t_{21}} & 
{\partial_2'+{\coth t_{21}+\coth t_{23}}}& -\frac{1}{\sinh t_{23}} \\
\frac{1}{\sinh t_{31}} & \frac{1}{\sinh t_{32}} &
{\partial_3'+{\coth t_{31} +\coth t_{32}}}
\end{pmatrix}.
\]

\end{proposition}
\begin{proof}
We consider the first row of the matrix (\ref{eqn:invdo2}):
\[
\E_1^\p \boldsymbol{e}_1+E_{12}^\p \boldsymbol{e}_2 +E_{13}^\p \boldsymbol{e}_3.
\]
Let $j=2$ or $3$. 
By Lemma \ref{lem:radial1}, we have
\[
E_{1j}^\p \boldsymbol{e}_j=
\coth t_{1j}\,\boldsymbol{e}_1 -\frac{1}{\sinh t_{1j}}\text{Ad}(a^{-1})E_{1j}^\kk \boldsymbol{e}_j.
\]
Since $\text{Ad}(a^{-1})E_{1j}^\kk$ acts on $V_{\tau^*}\simeq \C^3$ as the multiplication by 
${}^t E_{1j}^\kk=E_{j1}^\kk$,  $\text{Ad}(a^{-1})E_{1j}^\kk \boldsymbol{e}_1^*=\boldsymbol{e}_j^*$ by 
\cite[Table 2]{HOd}. Thus   
 the action  of the first row of  (\ref{eqn:invdo2}) on the coefficient of $\boldsymbol{e}_1^*\otimes\boldsymbol{e}_1$ is 
given by 
\[
(\E_1^\p+\coth t_{12}+\coth t_{13})\boldsymbol{e}_1^*\otimes \boldsymbol{e}_1
-\frac{1}{\sinh t_{12}}\boldsymbol{e}_2^*\otimes \boldsymbol{e}_2 
-\frac{1}{\sinh t_{13}}\boldsymbol{e}_3^*\otimes \boldsymbol{e}_3 .
\]
We can do similar computations for the second and third rows of the matrix (\ref{eqn:invdo2}) and 
obtain the matrix expression.



\end{proof}

\subsubsection{Radial part of the Casimir operator}

We put $E_{ij}^{(1)}=\frac12(E_{ij}^\p+E_{ij}^\kk)$ and 
$E_{ij}^{(2)}=\frac12(E_{ij}^\p-E_{ij}^\kk)$. For $1\leq i\not=j\leq 3$, $\{E_{ij}^{(1)}, \,E_{ij}^{(2)}\}$ forms  
a basis of the root space $\g^{e_i-e_j}$. 

The Casimir operator $\Omega_{\g_0}$ for $\g_0=\mathfrak{sl}(3,\C)$ is given by (\ref{eqn:casimir1a}). 
Let $\Omega^{(1)}$ and $\Omega^{(2)}$ denote the elements 
of $U((\g_0)_\C)$ that correspond to ${\Omega}_{\g_0}\otimes 1$ and 
$1\otimes {\Omega}_{\g_0}$ by the isomorphism (\ref{eqn:envisom}), respectively. 
That is, 
\begin{align*}
 {\Omega}^{(i)}
=&
 \frac19(E_{11}^{(i)}+E_{22}^{(i)}+E_{33}^{(i)})^2
-\frac13(E_{11}^{(i)}E_{22}^{(i)}+
E_{22}^{(i)}E_{33}^{(i)}+E_{33}^{(i)}E_{11}^{(i)}) 
\\
& +\frac{1}{6}(E_{12}^{(i)}E_{21}^{(i)}+E_{21}^{(i)}E_{12}^{(i)}+E_{13}^{(i)}
E_{31}^{(i)}+E_{31}^{(i)}E_{13}^{(i)}+E_{23}^{(i)}E_{32}^{(i)}+E_{32}^{(i)}E_{23}^{(i)})
\notag
\end{align*}
for $i=1,\,2$. 
The Casimir operator $\Omega$ for $(\g_0)_\C$ is given by 
$\Omega=\Omega^{(1)}+\Omega^{(2)}$. 

\begin{remark}
$\Omega^{(i)}$ ($i=1,2$) belongs to the center of the universal enveloping algebra 
of $\g_\C$. 
$-3\Omega^{(i)}+\frac13(E_{11}^{(i)}+E_{22}^{(i)}+E_{33}^{(i)})^2$ is denoted by 
$Cp_2^{(i)}$ in \cite[Section 5]{HOd}. 
\end{remark}

The following proposition gives the radial part of the Casimir operator. 
We take coordinate $(t_1,t_2,t_3)$ of $\aaa$ and put $t_{ij}=t_i-t_j$ as in \S~\ref{subsec:not2}. 

\begin{proposition}\label{prop:casimir2}
The $\tau$-radial part of $\Omega$ is given by 
\begin{align*}
-6R_\tau(\Omega) & +\frac13= \partial_1\partial_2+\partial_2\partial_3+\partial_3\partial_1 \notag \\
& -\{\coth t_{12}(\partial_1-\partial_2)+\coth t_{13}(\partial_1-\partial_3)+\coth t_{23}(\partial_2-\partial_3)\} \notag \\
&+\begin{pmatrix}
\frac{1}{\sinh^2 t_{12}}+\frac{1}{\sinh^2 t_{13}} & -\frac{\cosh t_{12}}{\sinh^2 t_{12}} & -\frac{\cosh t_{13}}{\sinh^2 t_{13}} \\
- \frac{\cosh t_{12}}{\sinh^2 t_{12}} & \frac{1}{\sinh^2 t_{12}}+\frac{1}{\sinh^2 t_{23}} & -\frac{\cosh t_{23}}{\sinh^2 t_{23}} \\
- \frac{\cosh t_{13}}{\sinh^2 t_{13}} & - \frac{\cosh t_{23}}{\sinh^2 t_{23}} &\frac{1}{\sinh^2 t_{13}}+\frac{1}{\sinh^2 t_{23}} \label{eqn:casimirrad2a}
\end{pmatrix}
\end{align*}
with respect to the basis \emph{(\ref{eqn:basisc2})} of $(V_{\tau^*}\otimes V_\tau)^{M_0}$. 

\end{proposition}
\begin{proof}
It follows from 
Lemma~\ref{lemma:radcas}
 that 
\begin{align*}
-6\,\Omega  = & -\frac13(E_{11}^\p+E_{22}^\p+E_{22}^\p)^2+E_{11}^\p E_{22}^\p+E_{22}^\p E_{33}^\p +E_{33}^\p E_{11}^\p \\
& -\frac13(E_{11}^\kk+E_{22}^\kk+E_{22}^\kk)^2
+E_{11}^\kk E_{22}^\kk+E_{22}^\kk E_{33}^\kk +E_{33}^\kk E_{11}^\kk \\
& -\sum_{1\leq i<j\leq 3}\coth (e_i-e_j)\,(E_{ii}^\p-E_{jj}^\p) \\
& +\frac12\sum_{1\leq i<j\leq 3}\frac{1}{\sinh^2 (e_i-e_j)}\
E_{ij}^\kk E_{ji}^\kk +\text{Ad}\,(a^{-1})(E_{ij}^\kk E_{ji}^\kk) \\
& \phantom{aaaaaaaaaaaaaaaaa}-2\cosh (e_i-e_j)\,
(\text{Ad}\,(a^{-1})E_{ij}^\kk) E_{ji}^\kk\}.
\end{align*}
Since $E_{ij}^\kk \boldsymbol{e}_k=\delta_{jk} \boldsymbol{e}_i$ and 
$E_{ij}^\kk \boldsymbol{e}_k^*=-\delta_{ik} \boldsymbol{e}_j^*$ for 
$1\leq i,\,j,\,k\leq 3$, the proposition follows. 
\end{proof}

\subsection{The case of $SU^*(6)/Sp(3)$}\label{subsec:H}
\subsubsection{Notation}
Let $G/K=SL(3,\mathbb{H})/SU(3,\mathbb{H})\simeq SU^*(6)/Sp(3)$. 
The Lie algebra $\mathfrak{sl}(3,\mathbb{H})$ is isomorphic to 
\[
\mathfrak{g}=\mathfrak{su}^*(6)=
\left\{
\begin{pmatrix}Z_1 & Z_2 \\ -\overline{Z}_2 & \overline{Z}_1\end{pmatrix}
\,:\,Z_1,\,Z_2\in M(3,\mathbb{C}),\,\text{Tr}\,Z_1+\text{Tr}\,\overline{Z}_1=0
\right\}.
\]
Define a Cartan involution by $\theta X=-{}^t \overline{X}$ for $X\in\mathfrak{g}$. Then we 
have the corresponding Cartan decomposition $\mathfrak{g}=\mathfrak{k}+\mathfrak{p}$ with 
\begin{align*}
& \mathfrak{k}=
\left\{
\begin{pmatrix}Z_1 & Z_2 \\ -\overline{Z}_2 & \overline{Z}_1\end{pmatrix}
\,:\,Z_1,\,Z_2\in M(3,\mathbb{C}),\,{}^t\overline{Z}_1=-Z_1,\,{}^t Z_2=Z_2
\right\} \simeq \mathfrak{sp}(3)
\\ & 
\mathfrak{p}
=\left\{
\begin{pmatrix}Z_1 & Z_2 \\ -\overline{Z}_2 & \overline{Z}_1\end{pmatrix}
\,:\,Z_1,\,Z_2\in M(3,\mathbb{C}),\,{}^t\overline{Z}_1=Z_1,\text{Tr}\,Z_1=0,\,
{}^t Z_2=-Z_2
\right\} .
\end{align*}
We have $\dim \,G/K=\dim \,\mathfrak{p}=14$. 

Complexifications of $\mathfrak{k}$ and $\mathfrak{p}$ are 
\begin{align*}
& \kk_\C=\mathfrak{sp}(3,\C)
=\left\{
\begin{pmatrix}A & B \\ C & -{}^t A\end{pmatrix}
\,:\,A,\,B,\,C\in M(3,\mathbb{C}),\,{}^t B=B,\,{}^t C=C
\right\} ,
\\ & 
\mathfrak{p}_\C
=\left\{
\begin{pmatrix}A & B \\ C & {}^t A\end{pmatrix}
\,:\,A,\,B,\,C\in M(3,\mathbb{C}),\,{}^t B=-B,\,{}^t C=-C,\,\text{Tr}\,A=0
\right\} .
\end{align*}

Let $E_{ij}$ denote the $3\times 3$ matrix with $ij$ entry $1$ and all other entries $0$. 
Define 
\begin{align*}
& Z_{ij}^{(1)}=\begin{pmatrix}E_{ij} & 0 \\ 0 &-E_{ji}  \end{pmatrix} \quad (1\leq i,\,j\leq 3),\\
& Z_{ij}^{(2)}=\begin{pmatrix} 0 & E_{ij}+E_{ji} \\ 0 & 0\end{pmatrix}\quad (1\leq i\leq j\leq 3), \\
& Z_{ij}^{(3)}=\begin{pmatrix} 0 & 0 \\ E_{ij}+E_{ji} & 0\end{pmatrix}\quad (1\leq i\leq j\leq 3).
\end{align*}
Then 
\[
\{Z_{ij}^{(1)}\,(1\leq i,\,j\leq 3),\,Z_{ij}^{(2)}\,(1\leq i\leq j\leq 3),\,
Z_{ij}^{(3)}\,(1\leq i\leq j\leq 3)\}
\]
forms a basis of $\kk_\C=\mathfrak{sp}(3,\C)$.
Define
\begin{align*}
& Y_{ij}^{(1)}=\begin{pmatrix}E_{ij} & 0 \\ 0 & E_{ji} \end{pmatrix}\quad (1\leq i,\,j\leq 3),\\
& Y_{ij}^{(2)}=\begin{pmatrix}0 & E_{ij}-E_{ji} \\ 0 & 0\end{pmatrix}\quad (1\leq i<j\leq 3), \\
& Y_{ij}^{(3)}=\begin{pmatrix} 0 & 0 \\ E_{ij}-E_{ji} & 0\end{pmatrix}\quad (1\leq i<j\leq 3).
\end{align*}
Then 
\begin{align*}
\{Y_{11}^{(1)}-Y_{22}^{(1)},\, & Y_{22}^{(1)}-Y_{33}^{(1)},\,
Y_{ij}^{(1)}\,(1\leq i\not=j\leq 3),\, \\
& Y_{ij}^{(2)}\,(1\leq i<j\leq 3),\,
Y_{ij}^{(3)}\,(1\leq i< j\leq 3)\}
\end{align*}
forms a basis of $\p_\C$.  
Let 
\[
\tilde{Y}_{ii}^{(1)}=Y_{ii}^{(1)}-\frac13(Y_{11}^{(1)}+Y_{22}^{(1)}+Y_{33}^{(3)})
\quad (1\leq i\leq 3).
\]
and 
\begin{equation}\label{eqn:coordaq}
\mathfrak{a}=\{t_1\tilde{Y}_{11}^{(1)}+t_2\tilde{Y}_{22}^{(1)}+t_3\tilde{Y}_{33}^{(1)}\,:\,t_1,\,t_2,\,t_3
\in\R,\,t_1+t_2+t_3=0\}.
\end{equation}
Then $\mathfrak{a}$ is a maximal abelian subspace of $\mathfrak{p}$. 
Let $e_i$ denote the linear form on $\mathfrak{a}$ defined by $e_i(Y_{jj}^{(1)})=\delta_{ij}$ $(1\leq i,\,j\leq 3)$. 
Then the restricted root system for $(\mathfrak{g},\mathfrak{a})$ is given by
\[
\Sigma=\{e_i-e_j\,:\,1\leq i\not=j\leq 3\}.
\]
Let $\Sigma^+$ denote the positive system defined by
\[
\Sigma^+=\{e_1-e_2,\,e_1-e_3,\,e_2-e_3\}.
\]
The root system $\Sigma$ is of type $A_2$ and the Weyl group $W$ is $S_3$.

For $i<j$ define
\begin{align*}
& X_{ij}^{(1)}=\begin{pmatrix}E_{ij} & 0\\ 0 & E_{ij}\end{pmatrix},\quad
 X_{ij}^{(2)}=\begin{pmatrix}\I E_{ij} & 0 \\ 0 & -\I E_{ij}\end{pmatrix},
\\ &
X_{ij}^{(3)}=\begin{pmatrix} 0 & -E_{ij}  \\  E_{ij} & 0 \end{pmatrix},
\quad 
X_{ij}^{(4)}=\begin{pmatrix} 0 & \I E_{ij} \\ \I E_{ij} & 0 \end{pmatrix}.
\end{align*}
Then $\{X_{ij}^{(1)}, X_{ij}^{(2)}, X_{ij}^{(3)}, X_{ij}^{(4)}\}$ forms a basis of $\mathfrak{g}_{e_i-e_j}$.


Define 
\[
\mathfrak{t}=\bigoplus_{i=1}^3 \R\I Z_{ii}^{(1)}.
\]
It is a Cartan subalgebra of $\kk$. The root system for $(\kk_\C,\ttt_\C)$ is 
\[
\Delta(\kk_\C,\ttt_\C)=\{\pm\varepsilon_i\pm \varepsilon_j\,(1\leq i<j\leq 3),\,
\pm 2\varepsilon_k\,(1\leq k\leq 3)\}, 
\]
where $\varepsilon_i$ is a linear forms on $\ttt_\C$ defined by $\varepsilon_i(Z_{jj}^{(1)})=\delta_{ij}\,(1\leq i,\,j\leq 3)$. 
 $Z_{ij}^{(1)}\,(i\not=j),\,Z_{ij}^{(2)},\,(i\leq j),\,Z_{ij}^{(3)}\,(i\leq j)$ are 
root vectors for $\varepsilon_i-\varepsilon_j,\,\varepsilon_i+\varepsilon_j,\,
-(\varepsilon_i+\varepsilon_j)\in \Delta(\kk_\C,\ttt_\C)$, respectively. 
Define positive system $\Delta(\kk_\C,\ttt_\C)^+$ by
\[
\Delta(\kk_\C,\ttt_\C)^+=\{\varepsilon_i\pm \varepsilon_j\,(1\leq i<j\leq 3),\,
 2\varepsilon_k\,(1\leq k\leq 3)\}.
\]

The adjoint representation of $\mathfrak{k}_\C$ on 
$\p_\C$ is an irreducible representation with the highest weight $\varepsilon_1+\varepsilon_2$. 
The subspace 
$\mathfrak{a}_\C\subset \mathfrak{p}_\C$ consists of zero weight vectors, $Y_{ij}^{(1)}$ $(i\not= j)$ is a 
weight vector with the weight $\varepsilon_i-\varepsilon_j$, and $Y_{ij}^{(2)},\,Y_{ij}^{(3)}$ are 
weight vectors with the weight $\varepsilon_i+\varepsilon_j,\,
-(\varepsilon_i+\varepsilon_j)$ respectively. 

Let $\m$ and $M$ denote the centralizer of $\aaa$ in $\kk$ and $K$ 
respectively.  Then $\m=\mathfrak{sp}(1)\oplus \mathfrak{sp}(1)\oplus \mathfrak{sp}(1)$, 
$M=Sp(1)\times Sp(1)\times Sp(1)$, and $\m_\C$ is generated by 
$Z_{ii}^{(1)},\,Z_{ii}^{(2)},\,Z_{ii}^{(3)}\,(1\leq i\leq 3)$. 

\subsubsection{Invariant differential operators}

Let $(\tau,V_\tau)$ be the standard six dimensional representation of $K=Sp(3)$. The 
highest weight of $\tau$ is $\varepsilon_1$ and weights of $\tau$ are $\pm\varepsilon_i$ 
$(1\leq i\leq 3)$. Let $\rho_i$ denote the $(i+1)$-dimensional irreducible representation of 
$Sp(1)\simeq SU(2)$. The restriction of $\tau$ to $M\simeq Sp(1)\times Sp(1)\times Sp(1)$ 
decomposes into two-dimensional irreducible representations 
$\rho_1\boxtimes \rho_0\boxtimes\rho_0$, $\rho_0\boxtimes \rho_1\boxtimes\rho_0$, 
and $\rho_0\boxtimes \rho_0\boxtimes\rho_1$ with highest weights $\varepsilon_1,\,
\varepsilon_2$, and $\varepsilon_3$, respectively. 

Since  $\tau|_M$ decomposes into multiplicity-free sum of irreducible representations,
  the algebra $\mathbb{D}(E_\tau)$ 
of invariant differential operators 
on the homogeneous vector bundle $E_\tau\rightarrow G/K$ associated with $\tau$
 is
commutative (\cite{D}). Moreover, the irreducible constituents of $\tau|_M$ are in a single 
$W$-orbits. 

\subsubsection{First order invariant differential operator}

We recall that representations of $K=Sp(6)$ on $\mathfrak{p}_\mathbb{C}$ and $V_\tau$ are 
irreducible representations with the highest weights $\varepsilon_1+\varepsilon_2$ and $\varepsilon_1$, respectively.  
By  the Littlewood-Richardson rule for $Sp(6)$, $\mathfrak{p}_\mathbb{C}\otimes V_\tau$ decomposes into multiplicity-free sum 
of three irreducible 
representations of $\mathfrak{k}_\C$ with highest weights $\varepsilon_1,\,
2\varepsilon_1+\varepsilon_2,\,\varepsilon_1+\varepsilon_2+\varepsilon_3$
(\cite{King}). 
Hence, the isotypic component of $\mathfrak{p}_\mathbb{C}\otimes V_\tau$ with the highest weight $\varepsilon_1$ gives 
a first order invariant differential operator on $E_\tau$, which is unique up to a constant multiple. 
We will give weight vectors of the isotypic component of $\mathfrak{p}_\mathbb{C}\otimes V_\tau$ with the highest weight $\varepsilon_1$. 

Let $\boldsymbol{e}_i$ be the $i$-th standard unit vector in $\mathbb{R}^6$, that is, 
its $i$-th component is one and other components are all zero. 
Then $\boldsymbol{e}_1,\boldsymbol{e}_2,\cdots, \boldsymbol{e}_6$ are 
weight vectors for the standard representation of $\mathfrak{sp}(3)$ on 
$\mathbb{C}^6$ with weights
 $\varepsilon_1,\varepsilon_2,\varepsilon_3,\,-\varepsilon_1,-\varepsilon_2,-\varepsilon_3$,  
respectively. 

\begin{lemma}\label{lem:wv}
The mapping $\boldsymbol{e}_j\mapsto w_j\,\,(1\leq j\leq 6)$ defined by the following formulae gives an $K$-equivariant 
injection from $V_\tau= \C^6$ into $\mathfrak{p}_\mathbb{C}\otimes V_\tau$. In particular, $w_1$ is a highest weight vector 
with the highest weight $\varepsilon_1$. 
\begin{align*}
& w_1=\tilde{Y}_{11}^{(1)}\otimes \boldsymbol{e}_1+Y_{12}^{(1)}\otimes \boldsymbol{e}_2+Y_{13}^{(1)}\otimes \boldsymbol{e}_3
+Y_{12}^{(2)}\otimes \boldsymbol{e}_5+Y_{13}^{(2)}\otimes \boldsymbol{e}_6,
\\ & 
w_2=Y_{21}^{(1)}\otimes \boldsymbol{e}_1+\tilde{Y}_{22}^{(1)}\otimes \boldsymbol{e}_2+Y_{23}^{(1)}\otimes \boldsymbol{e}_3
+Y_{21}^{(2)}\otimes \boldsymbol{e}_4+Y_{23}^{(2)}\otimes \boldsymbol{e}_6, \\
& w_3=Y_{31}^{(1)}\otimes \boldsymbol{e}_1+Y_{32}^{(1)}\otimes \boldsymbol{e}_2+\tilde{Y}_{33}^{(1)}\otimes \boldsymbol{e}_3
+Y_{31}^{(2)}\otimes \boldsymbol{e}_4+Y_{32}^{(2)}\otimes \boldsymbol{e}_5, \\
& w_4;=Y_{12}^{(3)}\otimes \boldsymbol{e}_2+Y_{13}^{(3)}\otimes \boldsymbol{e}_3+
\tilde{Y}_{11}^{(1)}\otimes \boldsymbol{e}_4+Y_{21}^{(1)}\otimes \boldsymbol{e}_5+Y_{31}^{(1)}\otimes \boldsymbol{e}_6, \\
& w_5=Y_{21}^{(3)}\otimes \boldsymbol{e}_1+Y_{23}^{(3)}\otimes \boldsymbol{e}_3+
Y_{12}^{(1)}\otimes \boldsymbol{e}_4+\tilde{Y}_{22}^{(1)}\otimes \boldsymbol{e}_5+Y_{32}^{(1)}\otimes \boldsymbol{e}_6, \\
& w_6=Y_{31}^{(3)}\otimes \boldsymbol{e}_1+Y_{32}^{(3)}\otimes \boldsymbol{e}_2+
Y_{13}^{(1)}\otimes \boldsymbol{e}_4+Y_{23}^{(1)}\otimes \boldsymbol{e}_5+\tilde{Y}_{33}^{(1)}\otimes \boldsymbol{e}_6.
\end{align*}

\end{lemma}
\begin{proof}
The subspace of $\mathfrak{p}_\mathbb{C}\otimes V_\tau$ with the weight $\varepsilon_1$ is spanned by 
\[
\tilde{Y}_{11}^{(1)}\otimes \boldsymbol{e}_1,\,\tilde{Y}_{22}^{(1)}\otimes \boldsymbol{e}_1,\,
Y_{12}^{(1)}\otimes \boldsymbol{e}_2,\,Y_{13}^{(1)}\otimes \boldsymbol{e}_3,\,
Y_{12}^{(2)}\otimes \boldsymbol{e}_5,\,Y_{13}^{(2)}\otimes \boldsymbol{e}_6.
\]
A highest weight vector with the highest weight $\varepsilon_1$ is a linear combination of the above 
six weight vectors that is killed by $Z_{12}^{(1)},\,Z_{23}^{(1)}$, and $Z_{33}^{(2)}$. 
Computing actions of these positive root vectors, we see that $w_1$ is a unique highest weight vector with the 
highest weight $\varepsilon_1$ up to a constant multiple. 
Other weight vectors are given by $w_2=Z_{21}^{(1)}\boldsymbol{e}_1,\,w_3=Z_{32}^{(1)}w_2,\,
w_4=Z_{13}^{(3)}w_3,\,w_5=-Z_{12}^{(1)}w_4,\,w_6=-Z_{23}^{(1)}w_5$. 
\end{proof}

We write a $C^\infty$-section $f$ of $E_\tau$ as $f=\sum_{i=1}^6 f_i\,\boldsymbol{e}_i\,\,(f_i\in C^\infty(G))$. Then the action of 
$w_1$ on $f$ is given by
\[
w_1f=\tilde{Y}_{11}^{(1)}f_1\,\boldsymbol{e}_1+Y_{12}^{(1)}f_2\,\boldsymbol{e}_2+Y_{13}^{(1)}f_3\, \boldsymbol{e}_3
+Y_{12}^{(2)}f_5\,\boldsymbol{e}_5+Y_{13}^{(2)}f_6\, \boldsymbol{e}_6.
\]
Actions of $w_2,\cdots,w_6$ are given in similar ways. 

By Lemma~\ref{lem:radial1}, we have
\begin{equation}
Y_{ij}^{(l)}=\coth t_{ij} Z_{ij}^{(l)}-\frac{1}{\sinh t_{ij}}\text{Ad}(a^{-1})Z_{ij}^{(l)}\,\,\, (1\leq i\not=j\leq 3,\,1\leq l\leq 3).
\end{equation}
Since $Z_{1j}^{(1)}\boldsymbol{e}_j=\boldsymbol{e}_1$ and $Z_{1j}^{(2)}\boldsymbol{e}_{4+j}=\boldsymbol{e}_1$ for $j=1,\,2$, 
it follows from Lemma~\ref{lem:wv} that 
\begin{align*}
w_1=(\tilde{Y}_{11}^{(1)}+ & 2\coth t_{12}+2\coth t_{13})\otimes\boldsymbol{e}_1 \\
& -\frac{1}{\sinh t_{12}}\text{Ad}(a^{-1})Z_{12}^{(1)}\otimes\boldsymbol{e}_2
-\frac{1}{\sinh t_{13}}\text{Ad}(a^{-1})Z_{13}^{(1)}\otimes\boldsymbol{e}_3  \\
& -\frac{1}{\sinh t_{12}} \text{Ad}(a^{-1})Z_{12}^{(2)}\otimes\boldsymbol{e}_5
-\frac{1}{\sinh t_{13}}\text{Ad}(a^{-1})Z_{13}^{(2)}\otimes \boldsymbol{e}_6
.
\end{align*}
In a similar way, we have
\begin{align*}
w_2=& \frac{1}{\sinh t_{12}}\text{Ad}(a^{-1})Z_{21}^{(1)}\otimes\boldsymbol{e}_1
 +(\tilde{Y}_{22}^{(1)}-2\coth t_{12}+2\coth t_{23})\otimes \boldsymbol{e}_2 \\
& -\frac{1}{\sinh t_{23}}\text{Ad}(a^{-1})Z_{23}^{(1)}\otimes\boldsymbol{e}_3
+\frac{1}{\sinh t_{12}}\text{Ad}(a^{-1})Z_{12}^{(2)}\otimes\boldsymbol{e}_4 \\
& -\frac{1}{\sinh t_{23}}\text{Ad}(a^{-1})Z_{23}^{(2)}\otimes\boldsymbol{e}_6, \\
w_3 =& \frac{1}{\sinh t_{13}}\text{Ad}(a^{-1})Z_{31}^{(1)}\otimes\boldsymbol{e}_1
+\frac{1}{\sinh t_{23}}\text{Ad}(a^{-1})Z_{32}^{(1)}\otimes\boldsymbol{e}_2 \\
&  +(\tilde{Y}_{33}^{(1)}-2\coth t_{13}-2\coth t_{23})\otimes \boldsymbol{e}_3 \\
& +\frac{1}{\sinh t_{13}} \text{Ad}(a^{-1})Z_{13}^{(2)}\otimes\boldsymbol{e}_4
+\frac{1}{\sinh t_{23}}\text{Ad}(a^{-1})Z_{23}^{(2)}\otimes \boldsymbol{e}_5, \\
w_4 =& -\frac{1}{\sinh t_{12}}\text{Ad}(a^{-1})Z_{12}^{(3)}\otimes\boldsymbol{e}_2
-\frac{1}{\sinh t_{13}}\text{Ad}(a^{-1})Z_{13}^{(3)}\otimes\boldsymbol{e}_3 \\
&  +(\tilde{Y}_{11}^{(1)}+2\coth t_{12}+2\coth t_{13})\otimes \boldsymbol{e}_4 \\
& +\frac{1}{\sinh t_{12}} \text{Ad}(a^{-1})Z_{12}^{(1)}\otimes\boldsymbol{e}_5
+\frac{1}{\sinh t_{13}}\text{Ad}(a^{-1})Z_{13}^{(1)}\otimes \boldsymbol{e}_6, \\
w_5=&  \frac{1}{\sinh t_{12}}\text{Ad}(a^{-1})Z_{12}^{(3)}\otimes\boldsymbol{e}_1
-\frac{1}{\sinh t_{23}}\text{Ad}(a^{-1})Z_{23}^{(3)}\otimes\boldsymbol{e}_3 \\
& -\frac{1}{\sinh t_{12}} \text{Ad}(a^{-1})Z_{12}^{(1)}\otimes\boldsymbol{e}_4 \\
&  +(\tilde{Y}_{22}^{(1)}-2\coth t_{12}+2\coth t_{23})\otimes \boldsymbol{e}_5 
+\frac{1}{\sinh t_{23}}\text{Ad}(a^{-1})Z_{23}^{(1)}\otimes \boldsymbol{e}_6, 
\\
w_6=&  \frac{1}{\sinh t_{13}}\text{Ad}(a^{-1})Z_{13}^{(3)}\otimes\boldsymbol{e}_1
+\frac{1}{\sinh t_{23}}\text{Ad}(a^{-1})Z_{23}^{(3)}\otimes\boldsymbol{e}_2 \\
& -\frac{1}{\sinh t_{13}} \text{Ad}(a^{-1})Z_{13}^{(1)}\otimes\boldsymbol{e}_4
-\frac{1}{\sinh t_{23}} \text{Ad}(a^{-1})Z_{23}^{(1)}\otimes\boldsymbol{e}_5 \\
&  +(\tilde{Y}_{33}^{(1)}-2\coth t_{13}-2\coth t_{23})\otimes \boldsymbol{e}_6.
\end{align*}

Since $f(ma)=f(am)$ for $m\in M$ and $a\in A$, $f$ vanishes outside
 $(V_{\tau^*}\otimes V_\tau)^M$. Since $M\simeq Sp(1)\times Sp(1)\times Sp(1)$ and 
the positive root for each $Sp(1)$ is $2\varepsilon_i\,\,(1\leq i\leq 3)$, 
\begin{equation}\label{eqn:bsq}
\{\boldsymbol{e}^*_4\otimes\boldsymbol{e}_1-\boldsymbol{e}^*_1\otimes\boldsymbol{e}_4, \,\,
\boldsymbol{e}^*_5\otimes\boldsymbol{e}_2-\boldsymbol{e}^*_2\otimes\boldsymbol{e}_5, \,\,
\boldsymbol{e}^*_6\otimes\boldsymbol{e}_3-\boldsymbol{e}^*_3\otimes\boldsymbol{e}_6\}
\end{equation}
forms a basis of $(V_{\tau^*}\otimes V_\tau)^M$. For $Z\in\mathfrak{k}_\C$, 
$\text{Ad}(a^{-1})Z$ acts on $f$ as $\tau^*(Z)$. It follows from the above expressions 
of $w_1,\cdots,w_6$, we have the following proposition. 

\begin{proposition}
With respect to the basis \emph{(\ref{eqn:bsq})}, the $\tau$-radial part  of the 
first order invariant differential operator given in Lemma~\ref{lem:wv} has the 
following matrix expression.
\begin{align*}
& \begin{pmatrix}
\partial_1' & 0 & 0 \\ 0 & \partial_2' & 0 \\ 0 & 0 & \partial_3'\end{pmatrix} \\
&  +\begin{pmatrix}
2(\coth t_{12}+\coth t_{13}) & -\frac{2}{\sinh t_{12}} & -\frac{2}{\sinh t_{13}} \\
\frac{2}{\sinh t_{12}} & {2(-\coth t_{12}+\coth t_{23})} & -\frac{2}{\sinh t_{23}} \\
\frac{2}{\sinh t_{13}} &\ \frac{2}{\sinh t_{23}} & {-2(\coth t_{13}+\coth t_{23})}
\end{pmatrix}.
\end{align*}
\end{proposition}

\subsubsection{Radial part of the Casimir operator}

The following proposition gives the $\tau$-radial part of $\Omega$. 
\begin{proposition}
The $\tau$-radial part  of the 
Casimir operator is given by
\begin{align*}
-12R_\tau(\Omega) & +3= \partial_1'\partial_2'+\partial_2'\partial_3'+\partial_3'\partial_1' \notag \\
& -2\{\coth t_{12}(\partial_1-\partial_2)+\coth t_{13}(\partial_1-\partial_3)+\coth t_{23}(\partial_2-\partial_3)\} \notag \\
&+2\begin{pmatrix}
\frac{1}{\sinh^2 t_{12}}+\frac{1}{\sinh^2 t_{13}} & -\frac{\cosh t_{12}}{\sinh^2 t_{12}} & -\frac{\cosh t_{13}}{\sinh^2 t_{13}} \\
- \frac{\cosh t_{12}}{\sinh^2 t_{12}} & \frac{1}{\sinh^2 t_{12}}+\frac{1}{\sinh^2 t_{23}} & -\frac{\cosh t_{23}}{\sinh^2 t_{23}} \\
- \frac{\cosh t_{13}}{\sinh^2 t_{13}} &  -\frac{\cosh t_{23}}{\sinh^2 t_{23}} &\frac{1}{\sinh^2 t_{13}}+\frac{1}{\sinh^2 t_{23}} \label{eqn:casimirrad2c}
\end{pmatrix}
\end{align*}
with respect to the basis \emph{(\ref{eqn:bsq})} of $(V_{\tau^*}\otimes V_\tau)^{M}$. 
\end{proposition}
\begin{proof}
It follows from  
Lemma~\ref{lemma:radcas} 
that 
\begin{align*}
-12\,\Omega  = & -\frac{1}{3}(Y_{11}^{(1)}+Y_{22}^{(1)}+Y_{33}^{(1)})^2+Y_{11}^{(1)}Y_{22}^{(1)}+Y_{22}^{(1)}Y_{33}^{(1)}+Y_{33}^{(1)}Y_{11}^{(1)} \\
& - \frac{1}{4}\sum_{i=1}^3 \{2 (Z_{ii}^{(1)})^2+Z_{ii}^{(2)}Z_{ii}^{(3)}+Z_{ii}^{(3)}Z_{ii}^{(2)}\} \\
& -2\sum_{1\leq i<j\leq 3}\coth (e_i-e_j)\,(Y_{ii}^{(1)}-Y_{jj}^{(1)}) \\
& +\frac12\sum_{1\leq i<j\leq 3}\frac{1}{\sinh^2 (e_i-e_j)}\{
Z_{ij}^{(1)}Z_{ji}^{(1)}+\text{Ad}\,(a^{-1})(Z_{ij}^{(1)}Z_{ji}^{(1)}) \\
& +Z_{ji}^{(1)}Z_{ij}^{(1)}+\text{Ad}\,(a^{-1})(Z_{ji}^{(1)}Z_{ij}^{(1)}) 
+Z_{ij}^{(3)}Z_{ji}^{(2)}+\text{Ad}\,(a^{-1})(Z_{ij}^{(3)}Z_{ji}^{(2)}) \\
& +Z_{ij}^{(2)}Z_{ji}^{(3)}+\text{Ad}\,(a^{-1})(Z_{ij}^{(2)}Z_{ji}^{(3)})\}\\
&  -\sum_{1\leq i<j\leq 3}\frac{\cosh(e_i-e_j)}{\sinh^2 (e_i-e_j)}
\{
Z_{ij}^{(1)}\text{Ad}\,(a^{-1})Z_{ji}^{(1)}+Z_{ji}^{(1)}\text{Ad}\,(a^{-1})Z_{ij}^{(1)} \\
& \phantom{bbbbbbbbbbbbbbbbbb}
+Z_{ij}^{(3)}\text{Ad}\,(a^{-1})Z_{ji}^{(2)}+Z_{ij}^{(2)}\text{Ad}\,(a^{-1})Z_{ji}^{(3)}
\}.
\end{align*}
The proposition follows by computing actions of $Z_{ij}^{(k)}$ 
$(1\leq i,\,j,\,k\leq 3)$ on weight vectors. 
\end{proof}

\subsection{Proof of Theorem~\ref{thm:sym}} 
In subsection~\ref{subsec:R}, \ref{subsec:C}, and \ref{subsec:H}, 
we have proved commutativity of $\mathbb{D}(E_\tau)$ and 
part (iii) of Theorem~\ref{thm:sym}. 
Moreover, we have seen that $\tau|_M$ decomposes into multiplicity free sum 
of three irreducible representations of $M$. 

We will show that 
$W\simeq M'/M$ acts transitively on three irreducible constituents of $\tau|_M$. 
We first consider the case of $\mathbb{K}=\R$. 
For $1\leq i<j\leq 3$, 
let ${s}_{ij}'\in M'$ denote representative of the reflection $s_{ij}\in W$ given by
\[
s_{12}'=\exp\frac{\pi}{2}K_3,\quad 
s_{23}'=\exp\frac{\pi}{2}K_1,\quad 
s_{13}'=\exp\frac{\pi}{2}K_2.
\]
The irreducible constituents of $\tau|_M$ are $\mathbb{R} \boldsymbol{e}_1,\,
\mathbb{R} \boldsymbol{e}_2,\,\mathbb{R} \boldsymbol{e}_3$. 
Since $s_{ij}'\R \boldsymbol{e}_k=\R\boldsymbol{e}_{s_{ij}k}$,  
$W=S_3$ acts transitively on the irreducible constituents of $\tau|_M$ 
as a permutation group. 
The case of $\mathbb{K}=\C$ is almost same as the case of $\mathbb{K}=\R$. 
In the case of $\mathbb{K}=\mathbb{H}$, 
the irreducible constituents of $\tau|_M$ are 
$\mathbb{R} \boldsymbol{e}_1\oplus \R\boldsymbol{e}_4,\,
\mathbb{R} \boldsymbol{e}_2\oplus\R \boldsymbol{e}_5,\,
\mathbb{R} \boldsymbol{e}_3\oplus \R \boldsymbol{e}_6$. 
For $1\leq i<j\leq 3$, we can take a representative of $s_{ij}\in W=S_3$ in $M'$ 
as 
$s'_{ij}=\exp\frac{\pi}{2}(Z_{ij}^{(1)}-Z_{ji}^{(1)})$ and we can show easily that 
$W=S_3$ acts transitively on the irreducible constituents of $\tau|_M$ 
as a permutation group. Thus part (i) of Theorem~\ref{thm:sym} is 
proved. 

Next, we  will prove part (ii) of Theorem~\ref{thm:sym}. First we consider the case of 
$\mathbb{K}=\R$. By Schur's lemma, 
$\text{End}_M(V_\tau)\simeq (V_{\tau^*}\otimes V_\tau)^M$ is a three dimensional 
vector space with basis $\{\boldsymbol{e}_1^*\otimes \boldsymbol{e}_1,\,
\boldsymbol{e}_2^*\otimes \boldsymbol{e}_2,\,\boldsymbol{e}_3^*\otimes \boldsymbol{e}_3\}$. 
The Weyl group $W=S_3$ acts on $ (V_{\tau^*}\otimes V_\tau)^M$ as a permutation group. 
Hence, the map from $(U(\mathfrak{a}_\C)\otimes \text{End}_M(V_\tau))^{M'}$ 
to $U(\mathfrak{a}_\C)$ defined by $\text{diag}(u_1,u_2,u_3)=u_3$ is an algebra 
isomorphism onto $U(\mathfrak{a}_\C)^{W_{e_1-e_2}}$. 

An element 
$D\in \mathbb{D}(E_\tau)$ acts  on 
the principal series representation of $G$ with the $K$-type $\tau$ and a 
parameter $\lambda\in\mathfrak{a}_\C^*$ as the multiplication by $\gamma_\tau(D)(\lambda)$. 
For the Casimir operator, we have
\[
\gamma_\tau(\Omega-\Omega_\mathfrak{m})(\lambda)=\langle\lambda,\lambda\rangle-
\langle\rho,\rho\rangle
\]
by \cite[Proposition 3.2]{MIO}. 
For the first order operator $D_1$, we have
\[
\gamma_\tau(D)(\lambda)=\text{diag}(\lambda_1,\lambda_2,\lambda_3)
\]
by 
\cite[Proposition 3.6]{MIO}. 
Since $\lambda_3$ and $\lambda_1\lambda_2+\lambda_2\lambda_3+\lambda_3\lambda_1$ are 
algebraically independent generators of $U(\mathfrak{a}_\C)^{W_{e_1-e_2}}
\simeq S(\mathfrak{a}_\C)^{W_{e_1-e_2}}$, 
 the algebra homomorphism 
(\ref{eqn:ghch}) is surjective, hence 
$\mathbb{D}(E_\tau)\simeq U(\mathfrak{a}_\C)^{W_{e_1-e_2}}$. 

We can prove  Theorem~\ref{thm:sym} (ii) for $\mathbb{K}=\C,\,\mathbb{H}$ in 
similar ways. 

\section{Commuting differential operators with $A_2$ symmetry}
\label{subsub:opgen}

\subsubsection{Matrix-valued commuting differential operators}

In Theorem~\ref{thm:sym}, we obtained matrix-valued 
commuting differential operators with $A_2$ symmetry as 
radial parts of invariant differential operators. 
We use the notation
\[
\partial_i=\displaystyle\frac{\partial}{\partial t_i},\quad t_{ij}=t_i-t_j
\]
as before. In the $GL$-picture, we have  differential 
operators $P_1,\,\tilde{Q}_1$, and $\tilde{P}_2$ given by
\begin{align*}
P_1 & =\partial_1+\partial_2+\partial_3, \\
\tilde{Q}_1 & =  
{\begin{pmatrix}
\partial_1 & 0 & 0 \\
0 & \partial_2 & 0 \\
0 & 0 & \partial_3
\end{pmatrix}}  \\
& \phantom{a}+{k}\begin{pmatrix}
{\scriptstyle\coth t_{12}+\coth  t_{13}} & -\frac{1}{\sinh t_{12}} & -\frac{1}{\sinh t_{13}}\\
\frac{1}{\sinh t_{12}} & {\scriptstyle-\coth t_{12}+\coth t_{23}}& -\frac{1}{\sinh t_{23}}\\
\frac{1}{\sinh t_{13}} &\frac{1}{\sinh t_{23}} & {\scriptstyle-\coth t_{13}-\coth t_{23}}
\end{pmatrix} 
, \\
\tilde{P}_2 & = {L_2}-4k^2 \\
& \phantom{a}+k  \begin{pmatrix}
\frac{1}{\sinh^2 t_{12}}+\frac{1}{\sinh^2 t_{13}} & -\frac{\cosh t_{12}}{\sinh^2 t_{12}} & -\frac{\cosh t_{13}}{\sinh^2 t_{13}} \\
 -\frac{\cosh t_{12}}{\sinh^2 t_{12}} & \frac{1}{\sinh^2 t_{12}}+\frac{1}{\sinh^2 t_{23}} & -\frac{\cosh t_{23}}{\sinh^2 t_{23}} \\
 -\frac{\cosh t_{13}}{\sinh^2 t_{13}} &  -\frac{\cosh t_{23}}{\sinh^2 t_{23}} & \frac{1}{\sinh^2 t_{13}}+\frac{1}{\sinh^2 t_{23}} 
\end{pmatrix}  , 
\end{align*}
where 
\[
L_2= \partial_1\partial_2+\partial_2\partial_3+\partial_3\partial_1
-k\sum_{1\leq i<j\leq 3} (\coth t_{ij})(\partial_i-\partial_j) ,
\]
and $P_1,\,\tilde{Q}_1$, and $\tilde{P}_2$ mutually commute for $k=1/2,\,1$, and $2$. 

Put
\[
{\delta_k^{1/2}}  =\prod_{1\leq i<j\leq 3}(\sinh t_{ij})^{k}. 
\]
The function $\delta_k$ is the density function for $G/K$ up to a constant multiple. 
Define $Q_1=\delta_k^{1/2}\circ \tilde{Q}_1\circ \delta_k^{-1/2}$ and  $P_2=\delta_k^{1/2}\circ \tilde{P}_2\circ \delta_k^{-1/2}$. 
By easy computations, we have the following corollary of Theorem~\ref{thm:sym}. 

\begin{cor}\label{cor:com0}
For $k=1/2,\,1$, and $2$, the differential operators $P_1,\,Q_1,\,P_2$ given by 
\begin{align*}
P_1 & =\partial_1+\partial_2+\partial_3, \\
{Q}_1 & =  \begin{pmatrix}
\partial_1 & 0 & 0 \\
0 & \partial_2 & 0 \\
0 & 0 & \partial_3
\end{pmatrix} +{k}\begin{pmatrix}
0 & -\frac{1}{\sinh t_{12}} & -\frac{1}{\sinh t_{13}}\\
\frac{1}{\sinh t_{12}} & 0 & -\frac{1}{\sinh t_{23}}\\
\frac{1}{\sinh t_{13}} &\frac{1}{\sinh t_{23}} & 0
\end{pmatrix} 
, \\
{P}_2 & =\partial_1\partial_2+\partial_2\partial_3+\partial_3\partial_1+{k(k-1)}\sum_{1\leq i<j\leq 3}{\frac{1}{\sinh^2 t_{ij}}} \\
& \phantom{a}+k  \begin{pmatrix}
\frac{1}{\sinh^2 t_{12}}+\frac{1}{\sinh^2 t_{13}} & -\frac{\cosh t_{12}}{\sinh^2 t_{12}} & -\frac{\cosh t_{13}}{\sinh^2 t_{13}} \\
 -\frac{\cosh t_{12}}{\sinh^2 t_{12}} & \frac{1}{\sinh^2 t_{12}}+\frac{1}{\sinh^2 t_{23}} & -\frac{\cosh t_{23}}{\sinh^2 t_{23}} \\
 -\frac{\cosh t_{13}}{\sinh^2 t_{13}} &  -\frac{\cosh t_{23}}{\sinh^2 t_{23}} & \frac{1}{\sinh^2 t_{13}}+\frac{1}{\sinh^2 t_{23}} 
\end{pmatrix}  
\end{align*}
mutually commute. 
\end{cor}

It is likely that the operators in Theorem~\ref{thm:sym} or Corollary~\ref{cor:com0} mutually commute for any $k\in\C$. 
Indeed, we can see by direct computations of commutators that $P_1,\,Q_1$, and $P_2$ mutually commute for any $k\in\C$. 
Moreover, the form of $P_1,\,Q_1$, and $P_2$ suggest the following generalization. 

Let $\beta(t)$ be an odd meromorphic function with a simple pole at the origin and define $P_1,\,Q_1$, and $P_2$ by 
\begin{align}
{P_1}=& {\partial_1+\partial_2+\partial_3}, \label{eqn:p1}\\
{P_2}=& 
{\partial_1\partial_2+\partial_2\partial_3+\partial_3\partial_1}
+k(k-1)\sum_{1\leq i<j\leq 3} \beta(t_{ij})^2 \label{eqn:p2}\\
+k & \begin{pmatrix}
\beta(t_{12})^2+\beta(t_{13})^2 & \beta'(t_{12}) & \beta'(t_{13}) \\
\beta'(t_{12}) & \beta(t_{12})^2+\beta(t_{23})^2 & \beta'(t_{23}) \\
\beta'(t_{13}) & \beta'(t_{23}) & \beta(t_{23})^2+\beta(t_{13})^2
\end{pmatrix}, \notag \\
{Q_1}= & 
\begin{pmatrix}
{\partial_1} & 0 & 0 \\
0 & {\partial_2} & 0 \\
0 & 0 & {\partial_3}
\end{pmatrix}
+k \begin{pmatrix}
0 & -\beta(t_{12}) & -\beta(t_{13}) \\
\beta(t_{12}) & 0 & -\beta(t_{23}) \\
\beta(t_{13}) & \beta(t_{23}) & 0
\end{pmatrix}. \label{eqn:q1}
\end{align}
By computing commutators, operators (\ref{eqn:p1}), (\ref{eqn:p2}), and (\ref{eqn:q1}) 
mutually commute if and only if $\beta$ satisfies the following functional equation
\begin{equation}\label{eqn:fe}
-\beta \left( s \right) \beta^2 \left( s+t \right)  
+\beta \left( s \right)  \beta^2 \left( t \right) 
\\  +
\beta \left( s+t \right)  \beta ' \left( t
 \right) + \beta ' \left( s+t \right) \beta
 \left( t \right) =0.
 \end{equation}
We know that $\beta(t)=1/\sinh t$ is a solution of the above functional equation. Moreover, since (\ref{eqn:fe}) 
does not depend on $k$, $P_1,\,P_2$, and $Q_1$ in Corollary~\ref{cor:com0} mutually commute for any $k\in \C$. 

Olshanetsky-Perelomov \cite[Appendix A]{OP} solved a functional equation that is essentially equivalent to 
(\ref{eqn:fe}) in their study of classical integrability of a system associated with a symmetric space. 
They proved that the general solution of (\ref{eqn:fe}) is given by 
\[ {\beta(t)=\displaystyle\frac{a}{\text{sn}(ax|\kappa)}} \quad (a\not=0),\]
where $\text{sn}$ is a Jacobi's elliptic function. Moreover, $\beta(t)$ is a real valued function on $\R$ if and 
only if $|a|=1,\,\,\kappa=1/a^2$ or $a\in \mathbb{R}\cup \sqrt{-1}\mathbb{R},\,\kappa\in\mathbb{R}$. 
Hence, we have the following theorem. 

\begin{theorem}\label{thm:com}
The operators \emph{(\ref{eqn:p1}), (\ref{eqn:p2})}, and \emph{(\ref{eqn:q1})} mutually commute for any $k\in\C$ and 
$\beta(t)=a/\text{\emph{sn}}(ax|\kappa)$ for any $a\not=0$ and $\kappa$. 
\end{theorem}

\subsection{$W$-equivariance of differential operators}
\label{subsubsec:invariance}
Differential operators in Theorem~\ref{thm:sym} and Corollary~\ref{cor:com0} are originally 
radial components on $\mathfrak{a}$ of differential operators on $G$ with certain $K$-equivariance, 
hence they have certain $W$-equivariance. 
 
Let $d$ be a $3\times 3$ matrix-valued differential operator on $\mathfrak{a}$. 
For $w\in S_3$, let $P_w$ denote the permutation matrix defined by ${P_w}=(\delta_{iw(j)})_{1\leq i,\,j\leq 3}$. 
Let $d^w$ denote the matrix-valued differential operator replacing $t$ by $w^{-1}t$ and 
$\partial_{j}\,(1\leq j\leq 3)$ by $\partial_{{w^{-1}(j)}}$ in $d$. Then our operators in Theorem~\ref{thm:sym}, 
Corollary~\ref{cor:com0}, and Theorem~\ref{thm:com} have $S_3$-invariance
\[
d^w=P_w^{-1}dP_w\quad (w\in S_3).
\]

\begin{remark}
In the case of $\beta(t)=1/t$, the first order operator $Q_1$ is essentially the Dunkl operator 
 and $P_2$ is essentially the Dunkl Laplacian for $A_2$ type 
root system (\cite{Dunkl}). We do not know whether there is an analogous 
relation with the Cherednik operator in trigonometric case. 
\end{remark}

In the forthcoming paper, we will discuss joint eigenfunctions of mutually commuting differential operators in 
Theorem~\ref{thm:com} in trigonometric case $\beta(t)=1/\sinh t$. In group case ($k=1/2,\,1,\,2$), $\tau$-radial part 
of a matrix coefficient of a principle series representation is a joint eigenfunction that is real analytic on $\mathfrak{a}$. 
In the case of $k=1/2$, Sekiguchi \cite{Sek} and Sono \cite{Sono} computed explicitly  Harish-Chandra's $c$-function. 
We will construct a real analytic joint eigenfunction for generic $k$, which is a vector-valued analogue of 
the Heckman-Opdam hypergeometric function. 

\subsection*{Acknowledgements}

The author would like to thank Professor Toshio Oshima, Professor Hiroyuki Ochiai, 
Professor Hiroshi Oda, and 
Professor Simon Ruijsenaars for helpful 
discussions and comments. 
The author would like to thank the referee for carefully reading the manuscript 
and for giving valuable comments.

\end{document}